%% file: template.tex
\documentclass{article}

\usepackage{arxiv}

\usepackage[utf8]{inputenc} % allow utf-8 input
\usepackage[T1]{fontenc}    % use 8-bit T1 fonts
\usepackage{hyperref}       % hyperlinks
\usepackage{url}            % simple URL typesetting
\usepackage{booktabs}       % professional-quality tables
\usepackage{amsfonts}       % blackboard math symbols
\usepackage{nicefrac}       % compact symbols for 1/2, etc.
\usepackage{microtype}      % microtypography
\usepackage{lipsum}		% Can be removed after putting your text content
\usepackage{graphicx}
\usepackage{algorithmic}
\usepackage{doi}

\usepackage[section]{placeins}

\usepackage{textalpha}
\usepackage{amssymb}

\usepackage{longtable}
\usepackage[ruled,vlined]{algorithm2e}
\usepackage{pgfplots}
\pgfplotsset{compat=newest}
\usepgfplotslibrary{fillbetween}
\usepackage{textalpha}
\usepackage{xcolor}

\usepackage[T1]{fontenc} % Use 8-bit encoding that has 256 glyphs

\usepackage[utf8]{inputenc} % Required for including letters with accents

\usepackage{graphicx} % Required for including images
\graphicspath{{Figures/}} % Set the default folder for images

\usepackage{enumitem} % Required for manipulating the whitespace between and within lists

\usepackage{lipsum} % Used for inserting dummy 'Lorem ipsum' text into the template

\usepackage{subfig} % Required for creating figures with multiple parts (subfigures)

\usepackage{amsmath,amsthm} % For including math equations, theorems, symbols, etc

\usepackage{varioref} % More descriptive referencing

\usepackage{tikz}
\usetikzlibrary{arrows,shapes,positioning,calc,3d}
\usepackage{longtable}
\usepackage{pgfplots}
\pgfplotsset{compat=newest}
\usetikzlibrary{pgfplots.fillbetween}
\usepgfplotslibrary{fillbetween}
\usepackage{placeins}
\theoremstyle{definition} % Define theorem styles here based on the definition style (used for definitions and examples)

\theoremstyle{plain} % Define theorem styles here based on the plain style (used for theorems, lemmas, propositions)
\newtheorem{theorem}{Theorem}

\theoremstyle{remark} % Define theorem styles here based on the remark style (used for remarks and notes)
\newtheorem*{remark}{Remark}

\newcommand{\R}{\mathbb{R}}

\newcommand{\x}{x}

\newcommand{\J}{J}

\newcommand{\eps}{\epsilon}
\newcommand{\Omeps}{\Omega_{\eps}}
\newcommand{\Om}{\Omega}

\newcommand{\Pv}{\vec{P}}
\newcommand{\nv}{\vec{n}}
\newcommand{\Vv}{\vec{V}}

\newcommand{\Qv}{\vec{Q}}
\newcommand{\rv}{\vec{r}}
\newcommand{\nab}{\nabla}

\newcommand{\Ga}{\Gamma_1}
\newcommand{\Gb}{\Gamma_2}
\newcommand{\Gc}{\Gamma_2}
\newcommand{\Gd}{\Gamma_2}
\newcommand{\Ge}{\Gamma_3}
\newcommand{\Lag}{\mathcal{L}}

\newcommand*\diff{\mathop{}\!\mathrm{d}}

\pgfdeclarelayer{ft}
\pgfdeclarelayer{bg}
\pgfsetlayers{bg,main,ft}
\newcommand\restr[2]{{% we make the whole thing an ordinary symbol
		\left.\kern-\nulldelimiterspace % automatically resize the bar with \right
		#1 % the function
		\vphantom{\big|} % pretend it's a little taller at normal size
		\right|_{#2} % this is the delimiter
}}

\newcommand{\expnumber}[2]{{#1}\mathrm{e}{#2}}
\title{Shape Optimization for the Mitigation of Coastal Erosion via Porous Shallow Water Equations}

\author{ \href{https://orcid.org/0000-0001-9930-065X}{\includegraphics[scale=0.06]{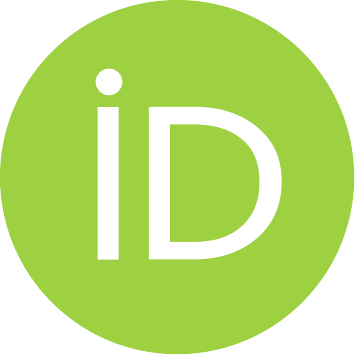}\hspace{1mm}Luka Schlegel} \\
	Department of Mathematics\\
	Universität Trier\\
	Universitätsring 15, 54296 Trier\\
	\texttt{schlegel@uni-trier.de} \\
	\And
	\href{https://orcid.org/0000-0001-7665-130X}{\includegraphics[scale=0.06]{orcid.pdf}\hspace{1mm}Volker Schulz} \\
	Department of Mathematics\\
	Universität Trier\\
	Universitätsring 15, 54296 Trier\\
	\texttt{volker.schulz@uni-trier.de} \\
}

\renewcommand{\headeright}{}
\renewcommand{\undertitle}{}
\renewcommand{\shorttitle}{SOMICE}

\hypersetup{
pdftitle={Porous SWE Paper},
pdfsubject={q-bio.NC, q-bio.QM},
pdfauthor={Luka Schlegel, Volker Schulz},
pdfkeywords={Shape Optimization, Obstacle Problem,  Numerical Methods,  Adjoint Methods, Porous Shallow Water Equations, Coastal Erosion},
}

\begin{document}
\maketitle

\begin{abstract}
Coastal erosion describes the displacement of land caused by destructive sea waves, currents or tides. Major efforts have been made to mitigate these effects using groynes, breakwaters and various other structures. We address this problem by applying shape optimization techniques on the obstacles. We model the propagation of waves towards the coastline using two-dimensional porous Shallow Water Equations with artificial viscosity. The obstacle's shape, which is assumed to be permeable, is optimized over an appropriate cost function to minimize the height and velocities of water waves along the shore, without relying on a finite-dimensional design space, but based on shape calculus. 
\end{abstract}

% keywords can be removed

\section{Introduction}
Coastal erosion describes the displacement of land caused by destructive sea waves, currents or tides. Major efforts have been made to mitigate these effects using groynes, breakwaters and various other structures. 
Among experimental set-ups to model the propagation of waves towards a shore and to find optimal wave-breaking obstacles, the focus has turned towards numerical simulations due to the continuously increasing computational performance. Calculating optimal shapes for various problems is a vital field, combining several areas of research. This paper builds up on the monographs \cite{Choi1987}\cite{Sokolowski1992}\cite{Delfour2011} to perform free-form shape optimization. In addition, we strongly orientate on \cite{Schulz2016}\cite{Schulz2014b} that use the Lagrangian approach for shape optimization, i.e. calculating state, adjoint and the deformation of the mesh via the volume form of the shape derivative assembled on the right-hand-side of the linear elasticity equation, as Riesz representative of the shape derivative. \\
Essential contributions to the field of numerical coastal protection have been made for steady \cite{Azerad2005}\cite{Mohammadi2008}\cite{Keuthen2015} and unsteady \cite{Mohammadi2011}\cite{Mohammadi2012}\cite{Schlegel20212} descriptions of propagating waves. In this paper we select one of the most widely applied system of wave equations. We describe the hydrodynamics by the set of Saint-Venant or better known as Shallow Water Equations (SWE), that originate from the famous Navier-Stokes Equations by depth-integrating, based on the assumption that horizontal length-scales are much larger than vertical ones \cite{SaintVenant1871}. To model a permeable obstacle, which can be exemplifying interpreted as a geotextile tube, a porosity parameter is introduced. Porous SWE models are being paid increasing attention throughout the last decade, mostly because its ability to perform large-scale urban flood modelling \cite{Dewals2021}. Over the years a variety of models have been introduced differing in terms of conceptual, mathematical and numerical aspects \cite{Soares2006}\cite{Sanders2008}\cite{Ozgen2016}.
Our model mainly builds up on \cite{Soares2006}, such that we are dealing with a single, depth-independent porosity parameter in the definition of the SWE. In addition, we restrict ourself to isotropic porosity effects, such that the parameter cannot account for directional effects, which forms a legitimate assumption for a geotextile obstacle.\\
We would like to highlight that porous SWE have been modelled mainly by techniques relying on constant cell approximations via Finite Volume schemes. In this paper we calculate and derive numerical solutions to porous SWE by high-order Discontinuous Galerkin (DG) methods. In this setting artificial viscosity is introduced to counter possible oscillations that can appear around a shock location and discretized using Symmetric Interior Penalty Discontinuous Galerkin (SIP-DG) \cite{Hartmann2008}. To deal with numerical difficulties, that arise due to the discontinuous material coefficient, we extend the notion of a well-balanced DG scheme for classical SWE with discontinuous sediment \cite{Wang2006} to porous, diffusive and two-dimensional SWE. In addition, it is noteworthy, that porous SWE have not been investigated in any kind of optimization yet, such that we firstly formulate adjoint and shape derivative for this set of equations and provide an algorithmic handle to this.\\
The paper is structured as follows: In Section \ref{sec:PF} we formulate the PDE-constrained optimization problem. In Section \ref{sec:DerSha} we derive the necessary tools to solve this problem, by deriving adjoint equations and the shape derivative in volume form. The final part, Section \ref{sec:NumRes}, will present numerical techniques and applications for a sample mesh such as a representative mesh for a real coastal section.

\section{Problem Formulation}\label{sec:PF}

Suppose we are given an open domain $\Om\subset\R^2$, which is split into the disjoint sets $\tilde{\Om},D\subset\Om$ such that $\tilde{\Om}\cup D\cup\Ge=\Om$ and $\Ga\cup\Gd=\partial\Om(=:\Gamma_{out})$. 
We assume the variable, interior boundary and the fixed outer $\partial\Om$ to be at least Lipschitz. One simple example of such kind is visualized below in Figure \ref{fig: domain}.

\begin{figure}[htb!]
	\centering
	\begin{tikzpicture}
	\begin{scope}
	\clip (0,0) rectangle (5,2.5);
	\draw (2.5,0) circle(2.5);
	\end{scope}

	\draw[->, >=stealth', shorten >=1pt] (2.33,0.7)   -- (2.0,1.1);
	\draw (0,0) --(5,0);
	\draw (2.5,0.5) circle (0.25);
	
	\node (A) at (3.5,1.2) {\large $\tilde{\Om}$};
	\node (B) at (2.5,-0.3) {$\Ga$};
	\node (E) at (2.5,2.8) {$\Gd$};
	\node (F) at (3.1,0.5) {$\Ge$};
	\node (G) at (2.5,0.5) {$D$};
	\node (H) at (2.05,0.7) {$\nv$};
	\end{tikzpicture}
	\caption[Illustrative Domain]{Illustrative Domain $\Om$ with Initial Circled Obstacle $D$ and Boundaries $\Ga,\Gd$ and $\Ge$}
	\label{fig: domain}
\end{figure}
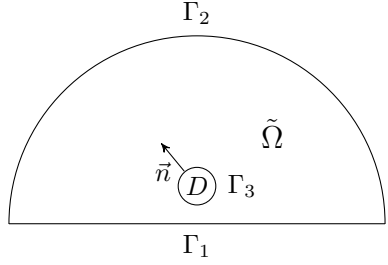

On this domain we model water wave and velocity fields as the solution to porous SWE with artificial viscosity. We interpret $\Ga,\Gd,\Ge$ as coastline, open sea and obstacle boundary and solve on $\Om\times(0,T)$
\begin{equation}
\begin{aligned}
&\partial_t(\phi U)+\nabla\cdot{(\phi F(U))}-\nab\cdot(G(f(\phi,\mu))\nab \hat{U})=\phi S(U)+S_\phi(U)\text{,}
\label{Eq:SWE}
\end{aligned}
\end{equation}
where we are given the SWE in vector notation with flux matrix 
\begin{equation}
\begin{aligned}
F(U)=\begin{pmatrix}
\Qv\\ 
\frac{\Qv}{H}\otimes \Qv+\frac{1}{2}gH^2\mathbf{I}_2
\end{pmatrix}
=
\begin{pmatrix}
Hu & vH\\	
Hu^2+\frac{1}{2}gH^2 & Huv \\
Huv & Hv^2 + \frac{1}{2}gH^2
\end{pmatrix}
\end{aligned}
\label{Eq:2Fluxes}
\end{equation}
for identity matrix $\mathbf{I}_2\in\R^{2\times2}$ and solution $U:\Om\times(0,T)\rightarrow\R\times\R^2$, where for simplicity the domain and time-dependent components are denoted by $U=(H,\Qv)=(H,Hu,Hv)$, with $H$ being the water height and $Hu,Hv$ the weighted horizontal and vertical discharge or velocity. For notational ease, we set $\hat{U}=(H+z,\Qv)$ for constant sediment height $z:\Om\times(0,T)\rightarrow\R$. The porosity is a scalar function $\phi:\Om\times(0,T)\rightarrow(0,1]$ representing the respective portion of space that is available to the flow. We define
\begin{align}
\phi\equiv\begin{cases}
\phi_1=\text{const.} \text{ in } \tilde{\Om}\times(0,T)\\
\phi_2=\text{const.} \text{ in } D\times(0,T)\text{.}
\end{cases}
\end{align}
The setting can be taken from Figure \ref{fig: Sketch}, where the region with varying porosity factor on $D$ is exemplifying highlighted in grey.
\begin{figure}[htb!]
\centering
\begin{tikzpicture}[scale=0.5, transform shape,x=0.5cm,y=0.5cm,z=0.3cm,>=stealth,grid/.style={very thin,gray}]

% The axes
\draw[->] (xyz cs:x=0) -- (xyz cs:x=13.5) node[above] {$x$};
\draw[->] (xyz cs:y=0) -- (xyz cs:y=13.5) node[right] {$z$};
\draw[->,opacity=0.4] (xyz cs:z=0) -- (xyz cs:z=13.5) node[above] {$y$};
% The thin ticks
\foreach \coo in {0,1,...,13}
{
	\draw (\coo,-1.5pt) -- (\coo,1.5pt);
	\draw (-1.5pt,\coo) -- (1.5pt,\coo);
	\draw[opacity=0.4] (xyz cs:y=-0.15pt,z=\coo) -- (xyz cs:y=0.15pt,z=\coo);
}
% The thick ticks
\foreach \coo in {0,5,10}
{
	\draw[thick] (\coo,-3pt) -- (\coo,3pt) node[below=6pt] {\coo};
	\draw[thick] (-3pt,\coo) -- (3pt,\coo) node[left=6pt] {\coo};
	\draw[thick,opacity=0.4] (xyz cs:y=-0.3pt,z=\coo) -- (xyz cs:y=0.3pt,z=\coo) node[below=8pt] {\coo};
}

% Now the waves
\draw[name path = A, black] (0,7) sin (2,8) cos(4,7) sin(6,6) cos(8,7) sin(10,8) cos(12,7);
\draw[name path = B, black] (0,7,4) sin (2,8,4) cos(4,7,4) sin(6,6,4) cos(8,7,4) sin(10,8,4) cos(12,7,4);
\draw[name path = C,black] (0,7) -- (0,7,4);
\draw[name path = D,black] (12,7) -- (12,7,4);
%\draw[blue] (0,7) sin (2,9);
\tikzfillbetween[of=A and B]{blue, opacity=0.3};

% Now the Sediment lower Level
\draw[name path = E, black] (0,0) -- (12,0);
\draw[name path = F, black,dashed,opacity=0.5] (0,0,4) -- (12,0,4);
\draw[name path = G,black,dashed,opacity=0.5] (0,0) -- (0,0,4);
\draw[name path = H,black] (12,0) -- (12,0,4);
%\tikzfillbetween[of=E and F]{brown, opacity=0.3};

% Sediment Upper Level
\draw[name path = I, black] (0,3) sin(2,2) cos(6,3) sin(10,4) cos(12,3);
\draw[name path = J, black] (0,3,4) sin(2,2,4) cos(6,3,4) sin(10,4,4) cos(12,3,4);
\draw[name path = K,black] (0,3) -- (0,3,4);
\draw[name path = L,black] (12,3) -- (12,3,4);
%\draw[blue] (0,7) sin (2,9);
\tikzfillbetween[of=I and J]{brown, opacity=0.3};
\tikzfillbetween[of=E and I]{brown, opacity=0.3};
%\tikzfillbetween[of=H and L]{brown, opacity=0.3};
\tikzfillbetween[of=A and I]{blue, opacity=0.1};
\tikzfillbetween[of=L and D]{blue, opacity=0.1};
\tikzfillbetween[of=H and L]{brown, opacity=0.3};

% Outer Edges
\draw[ black] (12,0) -- (12,7);
\draw[ black] (12,0,4) -- (12,7,4);
\draw[ black, dashed,opacity=0.5] (0,0,4) -- (0,7,4);

\draw [decorate,ultra thick,
decoration = {brace}] (12,3,2) -- (12,0,2) 
node[pos=0.5,right=4pt,black,thick]{$\mathbf{z}$} ;
\draw [decorate,ultra thick,
decoration = {brace}] (12,7,2) -- (12,3,2) 
node[pos=0.5,right=4pt,black,thick]{$\mathbf{H}$} ;

\tikzset{zxplane/.style={canvas is zx plane at y=#1,very thin}}
\tikzset{yxplane/.style={canvas is yx plane at z=#1,very thin}}

  \begin{scope}[canvas is zx plane at y=0]
\draw[name path = M,fill={rgb:black,1;white,2}] (2,8) circle (1cm);
%\node[black,thick] at (2,8) {\rotatebox{90}{\scalebox{7}[-1]{$\Large\mathbf{\phi_2}$}}};
%\node[black,thick] at (2,4) {\rotatebox{90}{\scalebox{7}[-1]{$\Large\mathbf{\phi_1}$}}};
\end{scope}

\end{tikzpicture}	
	\caption[Sketch of Wave and Sediment with Porous Region]{Cross-Section for Identification of Wave Height $H$, Sediment Height $z$ and Material Coefficient}
	\label{fig: Sketch}
\end{figure}
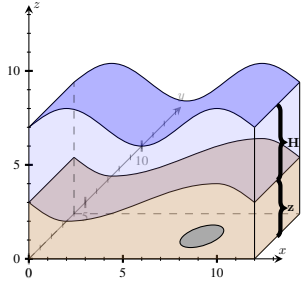

We define the first source term in (\ref{Eq:SWE}) as
\begin{equation}
\begin{aligned}\
S(U)=
\begin{pmatrix}
0\\
-gH\frac{\partial z}{\partial x}\\
-gH\frac{\partial z}{\partial y}
\end{pmatrix}
\end{aligned}\text{,}
\label{Eq:3Sources}
\end{equation}
responding to variations in the bed slope. In addition, the parameter $g$ represents the gravitational acceleration.
The second source term in (\ref{Eq:SWE}) corresponds to variations in the porosity coefficient and is chosen as \cite{Soares2006}
\begin{align}
S_\phi(U)=\begin{pmatrix}
0\\
g\frac{H^2}{2}\frac{\partial \phi}{\partial x}\\
g\frac{H^2}{2}\frac{\partial \phi}{\partial y}
\end{pmatrix}\text{.}
\end{align}
For the SWE we employ outer boundary conditions as rigid-wall and open sea boundary conditions for $\Ga$ and $\Gd$ as
\begin{equation}
\begin{aligned}
& \Qv\cdot \nv=0, \nab (H+z)\cdot \nv=0,\nab Q_1 \cdot \nv=0, \nab Q_2\cdot \nv=0 &\text{ on }& \Ga\times(0,T)&\\
&H=H_1, \nab Q_1\cdot \nv=0, \nab Q_2\cdot \nv=0 &\text{ on }&\Gd\times(0,T)&
\label{Eq:BCSWE}
\end{aligned}
\end{equation}
and transmissive interface conditions on $\Ge\times(0,T)$ for the continuity of the state
\begin{align} 
[\![H+z]\!]&=0
\label{BCSWEfirst}\\
[\![Q_1]\!]&=0\\
[\![Q_2]\!]&=0
\label{BCSWEfirst_2}
\end{align}
the diffusive flux
\begin{align} 
[\![\nab (H+z)\cdot\nv]\!]&=0\\
[\![\phi \nab (Q_1)\cdot\nv]\!]&=0\\
[\![\phi \nab (Q_2)\cdot\nv]\!]&=0
\end{align}
and the advective flux
\begin{align}
[\![\phi F(U)\cdot\nv]\!] = 0
\end{align}	
i.e.
\begin{align} 
[\![\phi \Qv\cdot\nv]\!]&=0\\
[\![\{\phi Q_1^2/H+1/2g\phi H^2; \phi Q_1Q_2/H\}\cdot\nv]\!]&=0\\
[\![\{\phi Q_1Q_2/H;\phi Q_2^2/H+1/2g\phi H^2\}\cdot\nv]\!]&=0
\label{Eq:BCSWElast}
\end{align}
for jump symbol on the interface $\Ge$ defined by $[\![H]\!]:=\restr{H}{\tilde{\Om}}-\restr{H}{\tilde{D}}$.
In addition, we prescribe to be determined initial conditions on $\Om\times\{0\}$ as
\begin{align}
U=U_0 \text{.}
\end{align}

\begin{remark}
	To prevent shocks or discontinuities and associated local oscillations that can appear in the original formulation of the hyperbolic SWE even for continuous data in finite time, diffusive terms are added in (\ref{Eq:SWE}) such that we obtain a set of fully parabolic equations. We control the amount of added diffusion by diagonal matrix $G(f(\phi,\mu))=\sum_{i=1}^ne_i^Tf(\phi,\mu) e_i e_i^T$ with entries $f(\phi,\mu)=(\mu_v,\phi\mu_{f}, \phi\mu_{f})\in\R_+^3$ and basis vector $e_i\in\R^n$ with $n$ being the number of dimensions in vector $f(\phi,\mu)$. In this setting $\mu_f$ is fixed to a small value, while we rely on shock detection in the determination of $\mu_v$ following \cite{Persson2006}. We refer to Section \ref{sec:exhalfcircled} and \ref{sec:mentawai} for more detailed information.
\end{remark}
\begin{remark}
	For classical SWE a physical interpretation can be obtained for the introduction of the viscous part in the conservation of momentum equation \cite{Oliger1978}. So far only porous SWE, without additional viscous terms, have been introduced in the literature, hence the usage is justified in Appendix \ref{app:ViscousPorMom}. Instead of the derived non-linear formulation we will work with linear diffusion in the sense of artificial viscosity, that is for stability also placed on the continuity equation. In this setting we follow the justification as in \cite{Guba2014}.  We would like to highlight that adjoint-based shape optimization for non-linear diffusion can be handled in the same way, leading to additional terms in the adjoint equations and the shape derivative. 
\end{remark}
\begin{remark}
A constant porosity coefficient $\phi_1=\phi_2$ in (\ref{Eq:SWE}) leads to SWE in the classical form, that are subject for adjoint-based shape optimization in \cite{Schlegel20212}. A detailed explanation for associated boundary conditions in the inviscid case can be found e.g. in \cite{Song2017}.
\end{remark}

We finally obtain a PDE-constrained optimization problem by constraining objective
\begin{align}
	J(\Om)=J_1(\Om)+J_2(\Om)+J_3(\Om)+J_4(\Om)\text{.}
\label{Eq:Obj_sum}
\end{align}
Here we try to meet certain predefined wave height and velocities $\bar{U}$ at the shore $\Ga$ weighted by diagonal matrix $N\in\R^{3\times3}$, such that we minimize objective $\J_1:\Om\rightarrow\mathbb{R}$, where
\begin{equation}
\begin{aligned} J_1(\Om)=&\int_{0}^{T}\int_{\Ga}\frac{1}{2}||N(\hat{U}(t,x)-\bar{U}(t,x))||_2^2\diff s\diff t\text{.}
\end{aligned}
\label{Eq:1Obj}
\end{equation}
This objective is supplemented by a volume penalty, which hinders the obstacle from becoming arbitrarily large
\begin{align}
J_2(\Om)=\nu_2\int_{D}1\diff x\text{,}
\end{align}
a perimeter regularization to ensure a sufficient regularity at obstacle level on $\Gc$, which lets us define necessary normal vectors, i.e.
\begin{align}
J_3(\Om)=\nu_3\int_{\Gamma_{3}}1\diff s
\label{Eq:1Peri}
\end{align}
and lastly a thickness control following \cite{Allaire2016}
\begin{align}
J_4(\Om)=\nu_4\int_{\Gamma_3}\int_0^{d_{min}}\left[(d_{\Om}\left(x-\xi\vec{n}(x)\right))^+\right]^2\diff\xi \diff s\text{.}
\label{Eq:1Thick}
\end{align}
Here $d_\Om$ represents the signed distance function with value 
\begin{align}
d_{\Om}(x)=\begin{cases}
d(x,\partial\tilde{\Om})\quad &\text{ if } x\in \tilde{\Om}\\
0 \quad &\text{ if } x\in \partial\tilde{\Om}
\\
-d(x,\partial\tilde{\Om})\quad &\text{ if } x\in \tilde{\Om}^c\text{,}
\end{cases}
\end{align}
where the Euclidian distance of $x\in\R^d$ to a closed set $K\subset\R^d$ is defined as
\begin{align}
d(x,K)=\min_{y\in K}||x-y||_2
\label{Eq:6EuclMin}
\end{align}
for Euclidian distance $||.||_2$. The three penalty terms are controlled by parameters $\nu_2,\nu_3$ and $\nu_4$, which need to be defined a priori (for further details cf. to Section \ref{sec:DerSha}).
\begin{remark}
The first objective (\ref{Eq:1Obj}) is of tracking type \cite{Reyes2015} to aim for rest-conditions of the water. Regions with comparable properties are known to mitigate sediment transport, e.g. as it can be seen in a coupling with equations of Exner-type \cite{exner1925}. Alternatively to (\ref{Eq:1Obj}) the minimization of the mechanical energy of destructive sea-waves may lead to further insights \cite{Azerad2005}.
\end{remark}

\section{Derivation of the Shape Derivative}\label{sec:DerSha}
We now fix notations and definitions in the first part, before deriving the adjoint equations and shape derivatives in the second part, that are necessary to solve the PDE-constrained optimization problem.
\subsection{Notations and Definitions}
In this section we introduce a methodology that is commonly used in shape optimization, extensively elaborated in various works \cite{Choi1987}\cite{Sokolowski1992}\cite{Delfour2011}. We fix notations and definitions following \cite{Schulz2016} and amend whenever it appears necessary.
We start by introducing a family of mappings $\{\phi_\eps\}_{\eps\in[0,\tau]}$ for $\tau>0$ that are used to map each current position $\x\in\Om$ to another by $\phi_\eps(\x)$, where we choose the vector field $\Vv$ as the direction for the so-called perturbation of identity
\begin{align}
\x_\eps=\phi_\eps(\x)=\x+\eps \Vv(\x)\text{.}
\label{Eq:4poi}
\end{align}
According to this methodology, we can map the whole domain $\Om$ to another $\Omeps$ such that 
\begin{align}
\Omeps=\{\x_\eps|x+\eps \Vv(x),x\in\Om\}\text{.}
\label{Eq:5domain}
\end{align}
We define the Eulerian Derivative as 
\begin{align}
DJ(\Om)[\Vv]=\lim_{\eps\rightarrow 0^+} \frac{J(\Om_\eps)-J(\Om)}{\eps}\text{.}\label{Eq:7EulerDer}
\end{align}
Commonly, this expression is called shape derivative of $J$ at $\Om$ in direction $\Vv$ and in this sense $J$ shape differentiable at $\Om$ if for all directions $\Vv$ the Eulerian derivative exists and the mapping $\Vv\mapsto DJ(\Om)[\Vv]$ is linear and continuous.
In addition, we define the material derivative of some scalar function $p:\Om\rightarrow\R$ at $x\in\Om$ by the derivative of a composed function $p_\eps\circ\phi_\eps:\Om\rightarrow\Om_\eps\rightarrow\R$ for $p_\eps:\Om_\eps\rightarrow\R$ as
\begin{align}
D_m p(x):=\lim_{\eps\rightarrow0^+}\frac{p_\eps\circ \phi_\eps(x)-p(x)}{\eps}=\frac{d}{d\eps}\restr{(p_\eps\circ \phi_\eps)(x)}{\eps=0^+}\label{Eq:8MatDer}
\end{align} 
and the corresponding shape derivative for a scalar $p$ and a vector-valued $\Pv$ for which the material derivative is applied component-wise as
\begin{align}
Dp[\Vv]:=D_mp-\Vv\cdot\nab p \label{Eq:9MatDer2}\\
D\Pv[\Vv]:=D_m\Pv-\Vv^T\nab \Pv \label{Eq:9MatDer2vec}\text{.}
\end{align}
In the following, we will use the abbreviation $\dot{p}$ and $\dot{P}$ to mark the material derivative of $p$ and $P$. In Section \ref{sec:DerSha} we will make use of the following calculation rules\cite{Berggren2010}
\begin{align}
D_m(pq)&=D_mpq+pD_mq\label{Eq:10MatProdR}\\
D_m\nab p&=\nab D_mp-\nab \Vv^T\nab p\label{Eq:11MatGradR}\\
D_m\nab \Pv&=\nab D_m\Pv-\nab \Vv^T\nab \Pv\label{Eq:11MatGradRvec}\\
D_m(\nab q^T\nab p)&=\nab D_mp^T\nab q-\nab q^T(\nab \Vv+\nab \Vv^T)\nab p+\nab p^T\nab D_mq\text{.} \label{Eq:12MatGradProdR}
\end{align}
The basic idea in the derivation of the shape derivative in the next section will be to pull back each integral defined on the on the transformed field back to the original configuration. Hence, we need to state the following rule for differentiating domain integrals  \cite{Berggren2010}
\begin{align}
	\frac{d}{d\eps}\restr{\left(\int_{\Om_{\eps}}p_\eps\right)}{\eps=0^+}=\int_\Om(D_mp+\nab\cdot \Vv p) \label{Eq:13DoaminR}\text{.}
\end{align}

\subsection{Shape Derivative}
We compute the adjoint equations and the shape derivative of the PDE-constrained optimization problem by formulating the Lagrangian
\begin{align}
\mathcal{L}(\Om,U,P) = J_1(\Om)+a(U,P)-b(P)\text{,} \label{Eq:16Lag}
\end{align}
where $J_1$ is objective (\ref{Eq:1Obj}), and $a(U,P)$ and $b(P)$ are obtained from boundary value problem (\ref{Eq:SWE}).
Here, we rewrite the equations in weak form by multiplying with some arbitrary test function $P\in H^1(\Om\times(0,T))^3$ obtaining the form $a(U,P)=a(H,\Qv,p,\rv)$, which is defined as
\begin{equation}
\begin{aligned}
a(H,\Qv,p,\rv):=&\int_0^T\int_\Om\left[\frac{\partial \phi H}{\partial t}+\nab\cdot(\phi\Qv)\right]p\diff x\diff t+\\
%---------------------------------
&\int_0^T\int_\Om\mu_v\nab (H+z)\cdot\nab p\diff x\diff t-\int_0^T\int_{\Ge}[\![\mu_v\nab(H+z)\cdot\nv p]\!]\diff s\diff t-\\
%---------------------------------
&\int_0^T\int_{\Gd}\mu_v\nab(H_1+z)\cdot\nv p\diff s\diff t+\\
%---------------------------------
&\int_0^T\int_\Om\left[\frac{\partial \phi\Qv}{\partial t}+\nab\cdot\left(\phi\frac{\Qv}{H}\otimes \Qv+\frac{1}{2}g\phi H^2\mathbf{I}_2\right)\right]\cdot\rv\diff x\diff t+\\
%---------------------------------
&\int_0^T\int_\Om \mu_f\phi\nab\Qv:\nab\rv\diff x\diff t-\int_0^T\int_{\Ge}[\![\phi\mu_f\nab\Qv\cdot\nv\cdot\rv]\!]\diff s\diff t+\\
&\int_0^T\int_\Om g\phi H\nab z\cdot\rv\diff x\diff t-
\int_0^T\int_\Om g\frac{H^2}{2}\nab \phi\cdot\rv\diff x\diff t
\label{Eq:17aweak}
\end{aligned}
\end{equation}
and a zero perturbation term.
\begin{remark}
	To deal with well-defined weak forms and to allow us to perform adjoint-based sensitivity analyses we assume the flow to be free of discontinuities, e.g. induced by a discontinuous bottom profile $z$ or wave height $H$. In addition, we need to employ a specific handle to the discontinuous porosity coefficient. In this paper we have used the strategy to write each integral over $\Omega$ as the sum over subdomains $\int_\Om=\int_{\tilde{\Om}}+\int_{D}$. In (\ref{Eq:17aweak}) and in what follows this decomposition is assumed.
\end{remark}
\begin{remark}
	For the discontinuous coefficient we could rely on a smoothed porosity controlled by $\alpha>0$, i.e. $\phi=\lim_{\alpha\rightarrow0}\phi_\alpha$, e.g. by using smoothed cell transitions or mollifiers. In this setting we could integrate over the whole domain $\Om$. Such a handle would call for the necessity to show convergence results for state, adjoint and shape derivative. Furthermore, we remark that a smoothing approach is presented in one dimension in Appendix \ref{app:NumericalConvAlpha}, where we have used a smoothed step-function. Here interface conditions would not be required in the continuous form.
\end{remark}

We obtain state equations from differentiating the Lagrangian w.r.t. $P$ and the auxiliary problem, the adjoint equations, from differentiating the Lagrangian with respect to the states $U$.
The adjoint is formulated in the following theorem:  
\begin{theorem}\label{AdjointTheo}
	(Adjoint) Assume that the parabolic PDE problem (\ref{Eq:SWE}) is $H^1$-regular, so that its solution $U$ is at least in $H^1(\Om\times(0,T))^3$. Then the adjoint in strong form with solution $P=(p,\rv)\in H^1(\Om\times(0,T))^3$ is given by
	\begin{equation}
	\begin{aligned}
	\phi\Big[-\frac{\partial p}{\partial t}+\frac{1}{H^2}(\Qv\cdot\nab)\rv\cdot\Qv-gH(\nab\cdot\rv) +g\nab z\cdot \rv\Big]-&\\
	\nab\cdot(\mu_v\nab p)-gH\nab\phi\cdot\rv&=-N_{11}((H+z)-\bar{H})_{\Ga}
	\end{aligned}
	\label{Eq:19Adjoint1}
	\end{equation}
	and
	\begin{equation}
	\begin{aligned}
	\phi\Big[-\frac{\partial\rv}{\partial t}-\nab p-\frac{1}{H}(\Qv\cdot\nab)\rv-\frac{1}{H}(\nab\rv)^T\Qv\Big]-&\\
	\nab\cdot(\mu_f\phi \nab\rv)&=-G(N_{22,33})(\Qv-\bar{\Qv})_{\Ga}
	\end{aligned}
	\label{Eq:19Adjoint2}
	\end{equation}
	with outer boundaries
	\begin{equation}
	\begin{aligned}
	p&=0 \quad & \text{ in }&&\Om&\times\{T\}&&\\
	\rv&=0 \quad & \text{ in }&&\Om&\times\{T\}&&\\
	\rv\cdot \nv=0, \nab p\cdot \nv=0,\nab r_1\cdot\nv=0, \nab r_2\cdot\nv&=0 \quad & \text{ on }&&\Ga&\times(0,T)&&\\
	\phi p\nv+\frac{\phi}{H_1}(\Qv\cdot\nv)\rv+\frac{\phi}{H_1}(\Qv\rv) \cdot \nv=0,\nab r_1\cdot\nv=0, \nab r_2\cdot\nv&=0 \quad & \text{ on }&&\Gd&\times(0,T)&&
	\end{aligned} 
	\label{Eq:20AdjointBC}
	\end{equation}
	and interface boundaries on $\Ge$ as
	\begin{align} 
	[\![p]\!]&=0\\
	[\![\rv]\!]&=0
	\end{align}
	such as
	\begin{align} 
	[\![\nab p\cdot\nv]\!]&=0\\
	[\![\phi\nab r_1\cdot\nv]\!]&=0\\
	[\![\phi\nab r_2\cdot\nv]\!]&=0
	\end{align}
	and
	\begin{align}
	[\![\phi F_U(P)\cdot\nv]\!] =0
	\end{align}	
	i.e.
	\begin{align} 
	[\![\phi(\frac{\Qv}{H^2}\cdot\rv\Qv+gH\rv)\cdot \nv]\!]&=0\\
	[\![\phi (p+2Q_1/Hr_1+Q_2/Hr_2;Q_2/Hr_1)\cdot\nv]\!]&=0\\
	[\![\phi(Q_1/Hr_1;p+2Q_2/Hr_2+Q_1/Hr_2)\cdot\nv]\!]&=0\text{.}
	\end{align}
\end{theorem}

\begin{proof}
	See Appendix \ref{app:derivadj}
\end{proof}
The porous SWE adjoint can be written in vector form as
\begin{align}
-\phi\frac{\partial P}{\partial t}+\phi AP_x+\phi BP_y+\phi CP-\nab\cdot(G(f(\phi,\mu))\nab P)=S\text{,} \label{Eq:AdjVectFormPor}
\end{align}
where
\begin{align}
A=\begin{pmatrix}
0&\frac{Q_1}{H^2}-gH&\frac{Q_1Q_2}{H^2}\\
-1&-2\frac{Q_1}{H}&-\frac{Q_2}{H}\\
0&0&-\frac{Q_1}{H}
\end{pmatrix},\quad	
B=\begin{pmatrix}
0&\frac{Q_1Q_2}{H^2}&\frac{Q_2^2}{H^2}-gH\\
0&-\frac{Q_2}{H}&0\\
-1&-\frac{Q_1}{H}&-2\frac{Q_2}{H}
\end{pmatrix}
\end{align}
and $C$ originates from variations in the sediment and the porosity such that
\begin{align}
C=\begin{pmatrix}
0&g\frac{\partial z}{\partial x}-g\frac{H}{\phi}\frac{\partial \phi}{\partial x}&g\frac{\partial z}{\partial y}-g\frac{H}{\phi}\frac{\partial \phi}{\partial y}\\
0&0&0\\
0&0&0
\end{pmatrix}\text{.}
\label{AdjVectorNotSource}
\end{align}
\begin{remark}
	In this paper, we will only consider the volume form of the shape derivative, which will be used to obtain smooth mesh deformations by a Riesz projection.
\end{remark}

\begin{theorem}\label{ShapeDerTheo}
	(Shape Derivative)
	Assume that the parabolic PDE problem (\ref{Eq:SWE}) is $H^1$-regular, so that its solution $U$ is at least in $H^1(\Om\times(0,T))^3$. Moreover, assume that
	the adjoint equations (\ref{Eq:AdjVectFormPor}) admit a solution $P\in H^1(\Om\times(0,T))^3$. Then the
	shape derivative of the objective $J_1$ at $\Om$ in the
	direction $\Vv$ is given by
	
	\begin{equation}
	\begin{aligned}
	DJ_1(\Om)[\Vv]=\int_{0}^{T}\int_{\Om}&\Big[-(\nab \Vv)^T:\nab (\phi\Qv) p - (\nab \Vv)^T:\nab \Qv\frac{\phi\Qv}{H}\cdot \rv- \\
	&(\nab \Vv\Qv\cdot\nab)\frac{\phi\Qv}{H}\cdot \rv -gH(\nab \Vv)^T\nab (\phi H) \cdot \rv-\\
	&\mu_v\nab (H+z)^T(\nab \Vv +\nab \Vv^T)\nab p - \\
	&\phi\mu_f \nab \Qv \nab \Vv:\nab \rv -\phi\mu_f \nab\Qv\nab \Vv^T:\nab\rv-  \\
	&g\phi H\nab \Vv^T\nab z\cdot\rv+\frac{1}{2}
	gH^2\nab \Vv^T\nab \phi\cdot\rv\\
	&div(\Vv)\{\frac{\partial \phi H}{\partial t}p+\nab\cdot (\phi\Qv) p+\frac{\partial \phi\Qv}{\partial t}\cdot \rv+\\
	& \phi(\Qv\cdot \nab)\frac{\Qv}{H}\cdot \rv +\nab\cdot (\phi\Qv)\frac{\Qv}{H}\cdot \rv +  \frac{1}{2}g\nab(\phi H^2)\cdot \rv +\\
	&g\phi H\nab z\cdot\rv+\mu_v\nab (H+z)\cdot\nab p +\\
	&\phi\mu_f \nab \Qv : \nab \rv -g\frac{1}{2}H^2\nab \phi\cdot\rv\}\Big]\diff x\diff t\text{.}
	\end{aligned}
	\label{Eq:21SD}
	\end{equation}
\end{theorem}

\begin{proof}
	See Appendix \ref{app:derivsha}
\end{proof}

The shape derivatives of the penalty terms (volume, perimeter and thickness) are obtained as, see e.g. \cite{Sokolowski1992}
\begin{align}
DJ_{2}(\Om)[\Vv]&=\nu_2\int_{D}\nab\cdot\Vv\diff x\\
DJ_{3}(\Om)[\Vv]&=\nu_3\int_{\Ge}\kappa\langle \Vv,\nv\rangle \diff s
\end{align}
and see \cite{Allaire2016} for
\begin{equation}
\begin{aligned}
DJ_4(\Om)[\Vv]=\nu_4\int_{\Ge}\int_0^{d_{min}}\Big[&\Vv(x)\cdot\vec{n}(x)\Big\{\kappa(x)(d_{\Om}\left(x_m\right)^+)^2+\\
&2d_{\Om}(x_m)^+\nab d_\Om(x_m)\cdot\nab d_\Om(x)\Big\}-\\
&\Vv(p_{\partial\Om}(x_m))\cdot\vec{n}(p_{\partial\Om}(x_m))2(d_{\Om}(x_m))^+\Big]\diff\xi \diff s
\end{aligned}
\end{equation}
for mean curvature $\kappa$ and offset point $x_m=x-\xi\nv(x)$, where we require the shape derivative of the signed distance function \cite{Delfour2011}
\begin{align}
Dd_\Om(x)[\Vv]=-\Vv(p_{\partial\Om}(x))\cdot\vec{n}(p_{\partial\Om}(x))\label{EqSDtoSDF}
\end{align}
with operator $p_{\partial\Om}$ that projects a point $x\in\Om$ onto its closest boundary and holds for all $x\notin\Sigma$, where $\Sigma$ is referred to as the ridge, where the minimum in (\ref{Eq:6EuclMin}) is obtained by two distinct points.

\section{Numerical Results} \label{sec:NumRes}
In the first part of this section we shortly sketch the SIP-DG method as in \cite{Hartmann2008}, before we discuss the well-balancedness of the porous SWE with diffusive terms and describing the algorithm for shape optimization in detail. Results are finally tested in two different scenarios.

\subsection{SIP-DG}
\label{Sec:SIPGDG}
As in \cite{Schlegel20212} we solve the boundary value problem (\ref{Eq:SWE}), the adjoint problem (\ref{Eq:AdjVectFormPor}) such as all quantities of the objective (\ref{Eq:Obj_sum}) with the finite element solver FEniCS \cite{FEniCS2015}. For the time discretization we can choose between implicit and explicit integration arising from theta-methods \cite{Hairer1996}.
High accuracy even for the inviscid and hyperbolic PDE is achieved using a SIP-DG method to discretize in space \cite{Hartmann2008}. This implies discontinuous cell transitions, and hence a formulation based on each element $\kappa\in\mathcal{T}_h$ or facet $\Gamma_I$ for a subdivision $\mathcal{T}_h$ of some domain $\Om$, such as a redefinition of each function and operator on the so-called broken and possibly vector-valued $d$-dimensional Sobolov space $\mathcal{H}^m(\mathcal{T}_h\times(0,T))^d$. In this light, we also need to define the average $\{\!\{U\}\!\}=(U^++U^-)/2$ and jump term $\underline{[\![U]\!]}=U^+\otimes n_++U^-\otimes n_-$ to express fluxes on cell transitions. The discretization then reads for solution and test-function $U_h,P_h\in\mathcal{H}^1(\mathcal{T}_h\times(0,T))^3$ as \cite{Hartmann2008}\cite{Houston2018}
\begin{equation}
\begin{aligned}
N_h(U_h,P_h)=&\int_{0}^{T}\int_{\Om}\Big[\frac{\partial \phi_h U_h}{\partial t}\cdot P_h-\phi_h F(U_{h}):\nab_hP_h+G(f(\phi_h,\mu))\nab_h(\hat{U}_h):\nab_hP_h-\\
&\phi_h S(U_h)\cdot P_h- S_{\phi}(U_h)\cdot P_h\Big] \diff x\diff t+\\
&\int_{0}^{T}\sum_{\kappa\in\mathcal{T}_h}\int_{\partial\kappa\setminus\Gamma}\mathcal{F}(U^{+}_h,U^{-}_h,\nv^+)\cdot P^{+}_h\diff s\diff t+\\
&\int_{0}^{T}\int_{\Gamma_I}\Big[\underline{\delta}_h:\underline{[\![P_h]\!]}\diff s\diff t-\{\!\{G(f(\phi_h,\mu))\nab_h (P_h)\}\!\}:\underline{[\![\hat{U}_h]\!]}-\\
&\{\!\{G(f(\phi_h,\mu))\nab_h( \hat{U}_h)\}\!\}:\underline{[\![P_h]\!]}\Big] \diff s\diff t+N_{\Gamma,h}(U_h,P_h,\phi_h)\text{,}
\label{Eq:22DGDiscretization}
\end{aligned}
\end{equation}
where fluxes at the discontinuous cell transitions are defined by the numerical flux function $\mathcal{F}(U^{+,*}_h,U^{-,*}_h,\nv^+)$. 
\begin{remark}For the advective flux and for a given flux Jacobian $J_i:=\partial_UF_i(U)$ and matrix $B(U,\nv)=\sum_{i=1}^{2}n_iJ_i(U)$, we can choose between a variety of fluxes \cite{Aizinger2002}. From here on the (Local) Lax-Friedrichs Flux is used that is defined as
	\begin{equation}
	\begin{aligned}
	\mathcal{F}(U^+,U^-,\nv)|\partial\kappa=\frac{1}{2}\left(F(U^+)\cdot \nv+F(U^-)\cdot \nv+\alpha_{\max}(U^+-U^-)\right)\text{,}	\end{aligned}
	\label{Eq:23LaxFrFlux}
	\end{equation}
	where $\alpha_{\max}=\max_{V=U^+,U^-}\{|\lambda(B(V,\nv_\kappa))|\}$ with $\lambda(B(V,\nv_\kappa))$ returning a sequence of eigenvalues for the matrix $B$ restricted on a side of element $\kappa$.
\end{remark}
\begin{remark}
	For the classical SWE eigenvalues of the SWE-Jacobian are obtained following \cite{Aizinger2002}, where $c=\sqrt{gH}$ denotes the wave celerity, as
	\begin{equation}
	\begin{aligned}
	\lambda(n_1J_1+n_2J_2)&=\{\lambda_{1},\lambda_2,\lambda_3\}
	\\&=\{un_1+vn_2-c,un_1+vn_2,un_1+vn_2+c\}\text{.}
	\end{aligned}
	\label{Eq:26EV}
	\end{equation}
\end{remark}
\begin{remark}
	We would like to highlight, that the chosen interface condition, can be resolved in an SIP-DG scheme for cells at the interface as well. Hence, the summation of all integrals of interior cell boundaries  $\sum_{\kappa\in\mathcal{T}_h}\int_{\partial\kappa\setminus\Gamma}$ in (\ref{Eq:22DGDiscretization}) includes the interface boundary $\Ge$. We show this in Appendix \ref{app:DGforInterface}.
\end{remark}
In (\ref{Eq:22DGDiscretization}) we define the penalization term for the viscous fluxes as \cite{Hartmann2008}
\begin{equation}
\begin{aligned}
\underline{\delta}_h(\hat{U}_h)=C_{IP}\frac{p_{DG}^2}{h_{\kappa}}\{\!\{G(f(\phi_h^+,\mu))\}\!\}\underline{[\![\hat{U}_h]\!]}\text{,}
\end{aligned}
\label{Eq:viscousflux}
\end{equation}
where $C_{IP}>0$ is a constant, $p_{DG}>0$ the polynomial order of the DG method and $h_\kappa>0$ the element-diameter for $\kappa\in\mathcal{T}_h$.
What is remaining in (\ref{Eq:22DGDiscretization}) is the specification of the boundary term, here we state that
\begin{equation}
\begin{aligned}
N_{\Gamma,h}(U_h,P_h)=&\int_{0}^T\int_{\Gamma}\mathcal{F}(U^{+}_h,U_\Gamma(U^{+}_h),\nv)\cdot P^{+}_h\diff s\diff t+\\
&\int_{0}^T\int_{\Gamma_n}\Big[\underline{\delta}_\Gamma (U_h^+):P_h^+\otimes \nv^++ G(f(\phi_h^+,\mu))\nab_h(\hat{U}^+_h):P_h^+\otimes\nv-\\
&G(f(\phi_h^+,\mu))\nab_hV_h^+:(\hat{U}_h^+-U_{\Gamma}(\hat{U}_h^+))\otimes \nv\Big] \diff s\diff t
\end{aligned}
\label{Eq:28DGDiscretizationBoundary}
\end{equation}
where $\Gamma_n$ are all boundaries of type Neumann. Additionally, we define
\begin{align}
\underline{\delta}_\Gamma(U_h^+)&=C_{IP}G(f(\phi_h^+,\mu))\frac{p_{DG}^2}{h_{\kappa}}(U_h^+-U_{\Gamma}(U_h^+))\otimes\nv\\
\mathcal{F}(U^{+}_h,U_\Gamma(U^{+}_h),\nv^+)&=\frac{1}{2}[\nv^+\cdot  F(U_h^+)+\nv^+\cdot  F(U_{\Gamma}(U_h^+))]\text{.}
\end{align} For the pure advective SWE open and rigid-wall boundary functions are defined as in \cite{Aizinger2002}.

\subsection{Well-Balancedness of SIP-DG for Porous SWE}
\label{Sec:WellBalancedness}
Approximate numerical solutions to systems like (\ref{Eq:SWE}), which allow to properly handle shocks and contact discontinuities, are in general known to be inaccurate even for near steady states \cite{Gosse2000}. This difficulty can be overcome by using the so-called well-balanced schemes, firstly introduced in \cite{Bermudez1994}. We will derive this property for our numerical scheme for the case of one-dimensional equations extending the approach in \cite{Wang2006} to porous SWE with diffusive terms. The two-dimensional formulation follows immediately then. Before starting, we explicitly state that we rely our solver on variables 
\begin{align}
\tilde{U}=\begin{pmatrix}
h\\
uh
\end{pmatrix}=
\begin{pmatrix}
\phi H\\
\phi uH
\end{pmatrix}\text{.}
\label{VarReform}
\end{align}
Hence, we redefine the $1$D porous SWE without diffusion in vector notation as
\begin{align}
\partial_t(\tilde{U})+\nabla\cdot({F(\tilde{U}))}=S(\tilde{U}) + S_\phi(\tilde{U}) 
\end{align}
for given flux matrix
\begin{align}
F(\tilde{U})=\begin{pmatrix}
hu\\ 
hu^2+\frac{1}{2}gh^2/\phi
\end{pmatrix}
\end{align}
and as before a source regarding the variations in the sediment
\begin{align}
S(\tilde{U})=\begin{pmatrix}
0\\
-gh\frac{\partial z}{\partial x}
\end{pmatrix}
\label{1Dsource}
\end{align}
and variations in the porosity factor
\begin{align}
S_\phi(\tilde{U})=\begin{pmatrix}
0\\
\frac{g}{2}\frac{h^2}{\phi^2}\frac{\partial \phi}{\partial x}
\end{pmatrix}\text{.}
\label{1Dporoussource}
\end{align}
Well-Balancing relies on incorporating the discretization of the source term in fluxes, such that e.g. (\ref{Eq:23LaxFrFlux}) used in (\ref{Eq:22DGDiscretization}) is redefined. Preserving still water stationary conditions means that $uh=0$ for $h/\phi+z = C$ for all $t\in (0,T)$. For the contribution to time changes it should be justified that on each element $\kappa\in\mathcal{T}_h=[x_{j-1/2},x_{j+1/2}]$ 
\begin{equation}
\begin{aligned}
R=&- \int_{\kappa}F(\tilde{U}_h(x,t))\cdot \partial_xP_h(x)\diff x+\mathcal{F}^L_{j+1/2}\cdot P_h(x^-_{j+1/2})-\mathcal{F}^R_{j-1/2}\cdot P_h(x^+_{j-1/2})\\
&-\int_{\kappa}S(\tilde{U}_h(x,t))\cdot P_h(x)\diff x-\int_{\kappa}S_{\phi}(\tilde{U}_h(x,t))\cdot P_h(x)\diff x=0\text{.}
\label{Eq. 20}
\end{aligned}
\end{equation}
In \cite{Wang2006} it is stated that Equation (\ref{Eq. 20}) is fulfilled if,
\begin{enumerate}
	\item $\mathcal{F}^L_{j+1/2}=F(\tilde{U}_h(x^-_{j+1/2}))$ and $\mathcal{F}^R_{j-1/2}=F(\tilde{U}_h(x^+_{j-1/2})$
	\item We are in a steady state and $u_h$ is a numerical approximation of $u$, hence 
	\begin{align*}
	\partial_xF(\tilde{U}_h)=\begin{pmatrix}
	0\\
	g(h_h,z_h,\phi_h)
	\end{pmatrix}\text{.}
	\end{align*}
\end{enumerate}
The assumption above can be easily justified and shows the appropriateness of the unmodified scheme in case of continuous piecewise sediment $z_h(x_{j+1/2}^-)=z_h(x_{j+1/2}^+)$ and porosity coefficients $\phi_h(x_{j+1/2}^-)=\phi_h(x_{j+1/2}^+)$. Situations with discontinuous sediment are dealt by relying on the idea of redefining variables \cite{Audusse2015}, i.e.
\begin{align}
h^{+,*}_{h,j+1/2}=\max\left(0,h^{+}_{h,j+1/2}+z^{+}_{h,j+1/2}-\max\left(z^{+}_{h,j+1/2},z^{-}_{h,j+1/2}\right)\right)\\
h^{-,*}_{h,j+1/2}=\max\left(0,h^{-}_{h,j+1/2}+z^{-}_{h,j+1/2}-\max\left(z^{+}_{h,j+1/2},z^{-}_{h,j+1/2}\right)\right)
\end{align}
which can be extended for varying porosity coefficient to
\begin{equation}
\begin{aligned}
&h^{+,*}_{h,j+1/2}=\\
&\max\left(0,\frac{h^{+}_{h,j+1/2}}{\phi^{+}_{h,j+1/2}}+z^{+}_{h,j+1/2}-\max(z^{+}_{h,j+1/2},z^{-}_{h,j+1/2})\right)\min\left(\phi^{+}_{h,j+1/2},\phi^{-}_{h,j+1/2}\right)
\end{aligned}
\end{equation}
\begin{equation}
\begin{aligned}
&h^{-,*}_{h,j+1/2}=\\
&\max\left(0,\frac{h^{-}_{h,j+1/2}}{\phi^{-}_{h,j+1/2}}+z^{-}_{h,j+1/2}-\max(z^{+}_{h,j+1/2},z^{-}_{h,j+1/2})\right)\min\left(\phi^{+}_{h,j+1/2},\phi^{-}_{h,j+1/2}\right)
\end{aligned}
\end{equation}
such that
\begin{align}
\tilde{U}^{+,*}_{h,j+1/2}=\begin{pmatrix}
h^{+,*}_{h,j+1/2}\\
uh^{+}_{h,j+1/2}
\end{pmatrix}\text{.}
\label{Eq24}
\end{align}
\begin{theorem} (Well-Balancedness)
	Redefining $\tilde{U}^{\pm,*}_{h,j+1/2}$ as in (\ref{Eq24}) in accordance with corrector-terms lead to a well-balanced scheme
\end{theorem}
\begin{proof}
	\begin{align*}
	\mathcal{F}^L_{j+1/2}=&\mathcal{F}(\tilde{U}^{-,*}_{h,j+1/2},\tilde{U}^{+,*}_{h,j+1/2})\\
	+&\begin{pmatrix}
	0\\
	\frac{g}{2}(h^{-}_{h,j+1/2})^2/\phi^{-}_{h,j+1/2}-\frac{g}{2}(h^{-,*}_{h,j+1/2})^2/\min\left(\phi^{+}_{h,j+1/2},\phi^{-}_{h,j+1/2}\right)
	\end{pmatrix}\\
	=&F(\tilde{U}^-_{h,j+1/2})
	\end{align*}
	Similarly
	\begin{align*}
	\mathcal{F}^R_{j-1/2}&=F(\tilde{U}^+_{h,j-1/2})
	\end{align*}
\end{proof}
Extending results to two dimensions can be done by looking at the residual on an element $\kappa\in\mathcal{T}_h$
\begin{equation}
\begin{aligned}
R=&- \int_{\kappa}F(\tilde{U}_h(x,t)):\nab P_h(x)\diff x+\int_{\partial \kappa}\mathcal{F}_{\partial \kappa}(U_h^+(x,t),U_h^-(x,t),\nv^+)\cdot P_h^+\diff s\\
&-\int_{\kappa}S(\tilde{U}_h(x,t))\cdot P_h(x)\diff x-\int_{\kappa}S_{\phi}(\tilde{U}_h(x,t))\cdot P_h(x)\diff x=0
\label{EResidualTwoD}
\end{aligned}
\end{equation}
and relying on a flux modification on each elemental boundary as
\begin{align}
&\mathcal{F}_{\partial \kappa}=\mathcal{F}(\tilde{U}^{-,*}_{h,\partial \kappa},\tilde{U}^{+,*}_{h,\partial \kappa},\nv^+_{\partial \kappa})\\
&+
G\begin{pmatrix}
0\\
\nv_{0,\partial \kappa}^+\\
\nv_{1,\partial \kappa}^+
\end{pmatrix}
\label{Eq:FluxMod2D1}\begin{pmatrix}
0\\
\frac{g}{2}(h^{+}_{h,\partial \kappa})^2/\phi^{+}_{h,\partial \kappa}-\frac{g}{2}(h^{+,*}_{h,\partial \kappa})^2/\min\left(\phi^{+}_{h,\partial \kappa},\phi^{-}_{h,\partial \kappa}\right)\\
\frac{g}{2}(h^{+}_{h,\partial \kappa})^2/\phi^{+}_{h,\partial \kappa}-\frac{g}{2}(h^{+,*}_{h,\partial \kappa})^2/\min\left(\phi^{+}_{h,\partial \kappa},\phi^{-}_{h,\partial \kappa}\right)
\end{pmatrix}\text{,}
\end{align}
where 
\begin{align}
\tilde{U}^{+,*}_{h,\partial \kappa}=\begin{pmatrix}
h^{+,*}_{h,\partial \kappa}\\
uh^{+}_{h,\partial \kappa}\\
vh^{+}_{h,\partial \kappa}
\end{pmatrix}\text{.}
\label{2dRedefinition}
\end{align} 
\begin{remark}
	Adding diffusive terms in the form
	\begin{align*}
	\partial_t(\tilde{U})+\nabla\cdot{( F(\tilde{U}))}-\nab\cdot(G(f(\phi,\mu))\nab \hat{U})=S(\tilde{U})+S_\phi(\tilde{U})\text{,}
	\end{align*}
	where $\hat{U}=(z+h/\phi,uh,vh)$, does not disturb well-balancedness, since rest conditions cancel contributing terms.
\end{remark}
\begin{remark}
	Reformulation (\ref{VarReform}) requires eigenvalues in the form of (\ref{Eq:26EV}) with $c=\sqrt{gh/\phi}$ to be used in the numerical flux function.
\end{remark}
\begin{remark}
	The numerical scheme used to handle discontinuous sediment and porosity coefficients forms the limit of a smoothed scenario, such that $U_\alpha\rightarrow U$ in $H^1(\Om\times(0,T))^3$ where $\phi_\alpha\rightarrow\phi$ for $\alpha\rightarrow0$ which is verified for one dimension numerically in Appendix \ref{app:NumericalConvAlpha}.
\end{remark}

\subsection{Implementation Details for Shape Optimization}
\label{Sec:Implementation}
We rely on the classical structure of adjoint and gradient-descent based shape optimization algorithms shortly sketched in the algorithm below.

\begin{algorithm}
\caption{Shape Optimization Algorithm}
\begin{algorithmic}
  \STATE Initialization
  \WHILE  {$||DJ(\Om_k)[\Vv]||>\eps_{TOL}$}
    \STATE 1. Calculate SDF $w_k$
    \STATE 2. Calculate State $U_k$ via $\tilde{U}_k$
    \STATE 3. Calculate Adjoint $P_k$ via $\tilde{P}_k$
    \STATE 4. Use $DJ_{1,2,3,4}(\Om_k)[\Vv]$ to calculate Gradient $W_k$
    \STATE 5. Perform Linesearch for $\tilde{W}_k$
    \STATE 6. Deform $\Om_{k+1}\longleftarrow \phi_{\tilde{W}_k}(\Om_k)$
  \ENDWHILE
\end{algorithmic}
\end{algorithm}

The signed distance function in (\ref{Eq:1Thick}) is calculated as the solution to the Eikonal Equation with $f(x)=1$, $q(x)=0$
\begin{equation}
\begin{aligned}
|\nab w(x)|&=f(x) \quad &x\in\Om&&&&\\
w(x)&=q(x)\quad &x\in\partial\tilde{\Om}&\text{,}&&&
\end{aligned}
\end{equation}
where we solve a stabilized viscous version to obtain $w\in H^1(\Om)$ for all $v\in H^1(\Om)$ i.e.
\begin{equation}
\begin{aligned}\int_\Om\sqrt{\nab w\cdot\nab w}v\diff x-\int_\Om fv\diff x +\int_\Om\mu_{SDF} \nab w\cdot\nab v\diff x=0\text{,}
\end{aligned}
\end{equation}
where $\mu_{SDF}=\max h_\kappa$ is dependent on the element-diameter $h_\kappa$.
Numerical solutions to the adjoint equations require us to rewrite the vector form (\ref{Eq:AdjVectFormPor}) with the help of the product rule, i.e.
\begin{align}
\frac{\partial P}{\partial t}-\nab\cdot(AP,BP)-\tilde{C}P+\nab\cdot(G(f(\phi,\mu))\nab P)=-S \text{,}
\end{align}
where $\tilde{C}$ is defined to be
\begin{align}
\tilde{C}=C-A_x-B_y\text{.}
\end{align} 
As we have shown in \cite{Schlegel20212} the eigenvalues of matrix $B^*(P,\nv)$ belonging to the adjoint flux Jacobian $\mathcal{J}^*_i:=\partial_PF^*_i(P)$ equal the eigenvalues of matrix $B(U,\nv)$ belonging to the flux Jacobian $\mathcal{J}_i:=\partial_UF_i(U)$. The SWE adjoint problem is then solved in the same manner as the scheme for the forward system (\ref{Sec:SIPGDG}) with (\ref{Sec:WellBalancedness}),  using a well-balanced SIP-DG discretization in space and a member of the theta-methods for the time discretization.

The finite element mesh $\mathcal{T}_h$ deforms in each iteration via the solution $\vec{W}:\Om\rightarrow\R^2$ of the linear elasticity equation \cite{Schulz2016}
\begin{equation}
\begin{aligned}
\int_\Om\sigma(\vec{W}):\eps(\Vv)\diff x&=DJ(\Om)[\Vv] \hspace{1cm} \quad &\forall\Vv\in H_0^1(\Om,\R^2)\\
\sigma:&=\lambda_{elas} Tr(\eps(\vec{W}))I+2\mu_{elas}\eps(\vec{W})\\
\eps(\vec{W}):&=\frac{1}{2}(\nab \vec{W}+\nab \vec{W}^T)\\
\eps(\Vv):&=\frac{1}{2}(\nab\Vv+\nab\Vv^T)\text{,}
\end{aligned}
\label{Eq:29LinearElasticity}
\end{equation}
where $\sigma$ and $\eps$ are called strain and stress tensor and $\lambda_{elas}$ and $\mu_{elas}$ are called Lamé parameters. We have chosen $\lambda_{elas}=0$ and $\mu_{elas}$ as the solution of the following Poisson Problem
\begin{equation}
\begin{aligned}
-\bigtriangleup\mu&=0 \hspace{1cm} &&\text{in }&& \Om&&\\
\mu&=\mu_{max} \hspace{1cm} &&\text{on }&& \Ge&&\\
\mu&=\mu_{min} \hspace{1cm} &&\text{on }&& \Ga, \Gb\text{.}&&
\end{aligned}
\label{Eq:33Lame}
\end{equation}
The source term $DJ(\Om)[\Vv]$ in (\ref{Eq:29LinearElasticity}) consists of a volume and surface part, i.e. $DJ(\Om)[\Vv]=DJ_\Om[\Vv] + DJ_{\Ge}[\Vv]$. 
\begin{remark}
	The volumetric share comes from our porous SWE shape derivative and the volume penalty. Before assembling, the test vector fields whose support does not intersect with the interface $\Ge$ are set to zero \cite{Schulz2016}. The surface part comes from the parameter regularization and the thickness constraint in (\ref{Eq:Obj_sum}).
\end{remark}
\begin{remark}
	To guarantee the attainment of useful shapes, which minimize the objective, a backtracking line search is used, which limits the step size in case the shape space is left \cite{Schulz2016} i.e. having intersecting line segments or in the case of a non-decreasing objective evaluation. As described in the algorithm above, the iteration is finally stopped if the norm of the shape derivative has become sufficiently small.
\end{remark}

\subsection{Example: The Half-Circled Mesh}\label{sec:exhalfcircled}
In the first example, we will look at the model problem - the half circle that was described in Section \ref{sec:PF}. As before, we interpret $\Ga,\Gd,\Ge$ as coastline, open sea and obstacle boundary. We will work with a rest height of the water at $\bar{H}=1$, while targeting zeroed velocities. We penalize volume and thinness by setting $\nu_2=\expnumber{1}{-4}$, $\nu_4=\expnumber{1}{-2}$ and enforce a stronger regularization by $\nu_3=\expnumber{1}{-4}$. The parameters in the porous shallow water system are set as follows: For the weight of the diffusion terms in the momentum equation we set $\mu_f=\expnumber{1}{-2}$ and determine $\mu_v$ by the usage of the mentioned shock detector \cite{Persson2006}. The gravitational acceleration is fixed at roughly $9.81$. The mesh $\mathcal{T}_h$ displayed in Figure \ref{fig:InitHalfCircle} was created using the finite element mesh generator GMSH \cite{Geuzaine2009}, where the vertex density around the obstacle is increased to ensure a high resolution. 
\begin{figure}[htb!]
	\centering
	\begin{tikzpicture}
	\node[anchor=south west,inner sep=0] (1) {\vspace{-0.9cm}
	\includegraphics[scale=0.13]{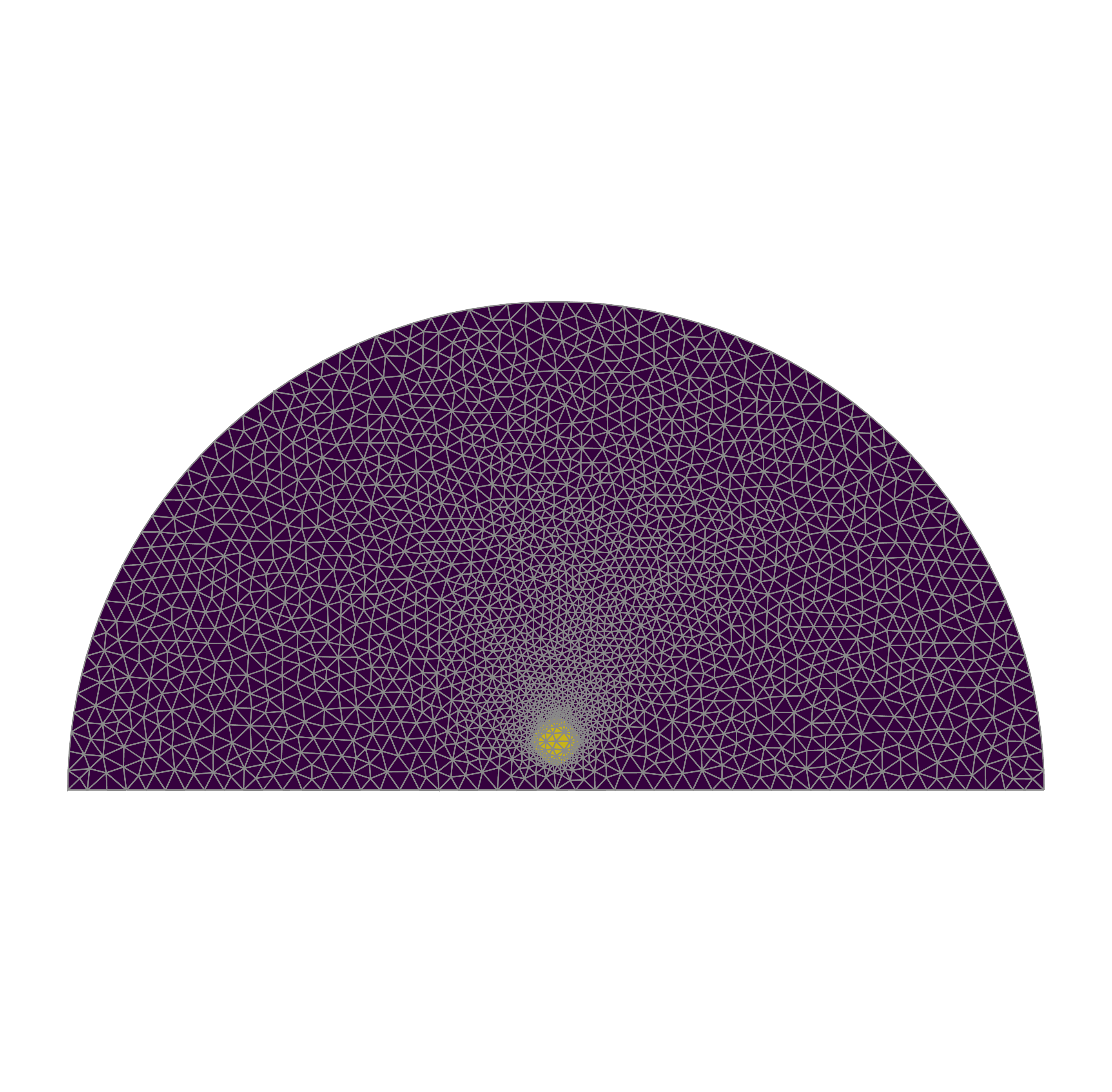}};
	\node[above= -1.1cm of 1] (7) 
	{\tiny Initial Mesh \& Porosity};
	\node[right= -0.1cm of 1] (7) 
	{\includegraphics[scale=0.07]{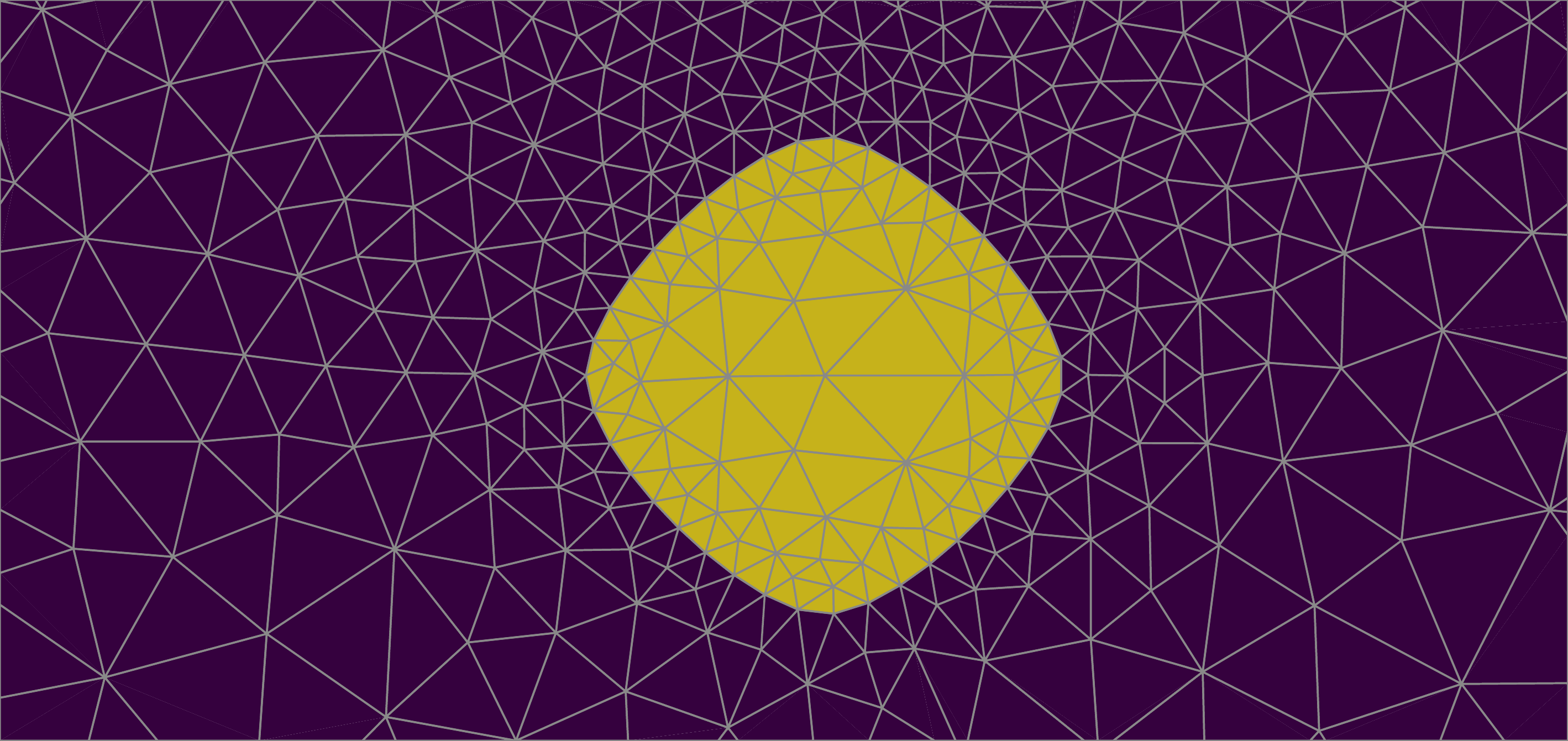}};
	\node[above left = -1cm and -0.01cm of 1](2)
	{\tiny 1.};
	\node[below right = -1cm and -1cm of 1](3)
	{\includegraphics[width=4.5cm,height=3.5cm,scale=0.5]{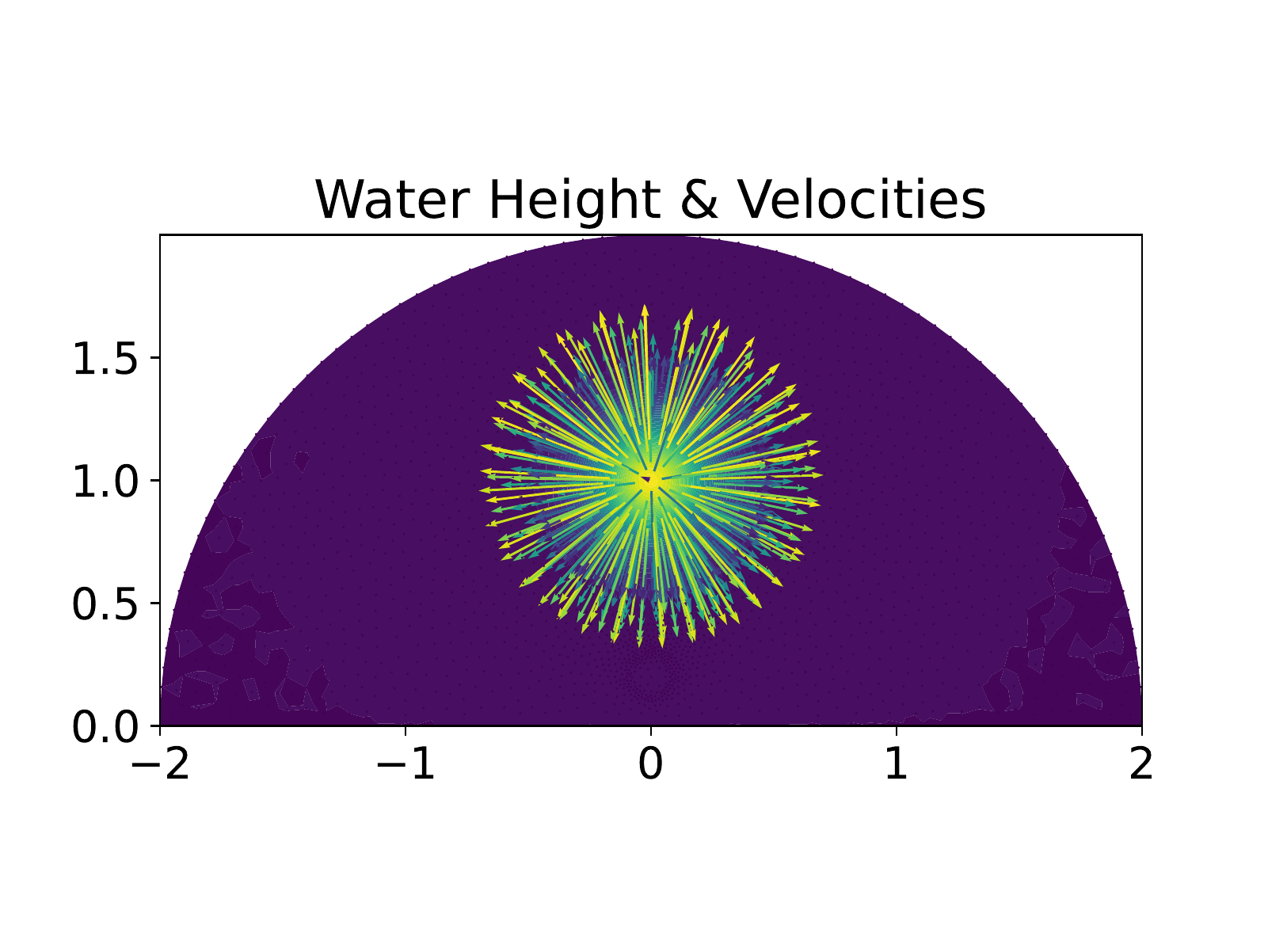}};
	\node[left= 0.015cm of 3](4)
	{\includegraphics[width=4.5cm,height=3.5cm,scale=0.5]{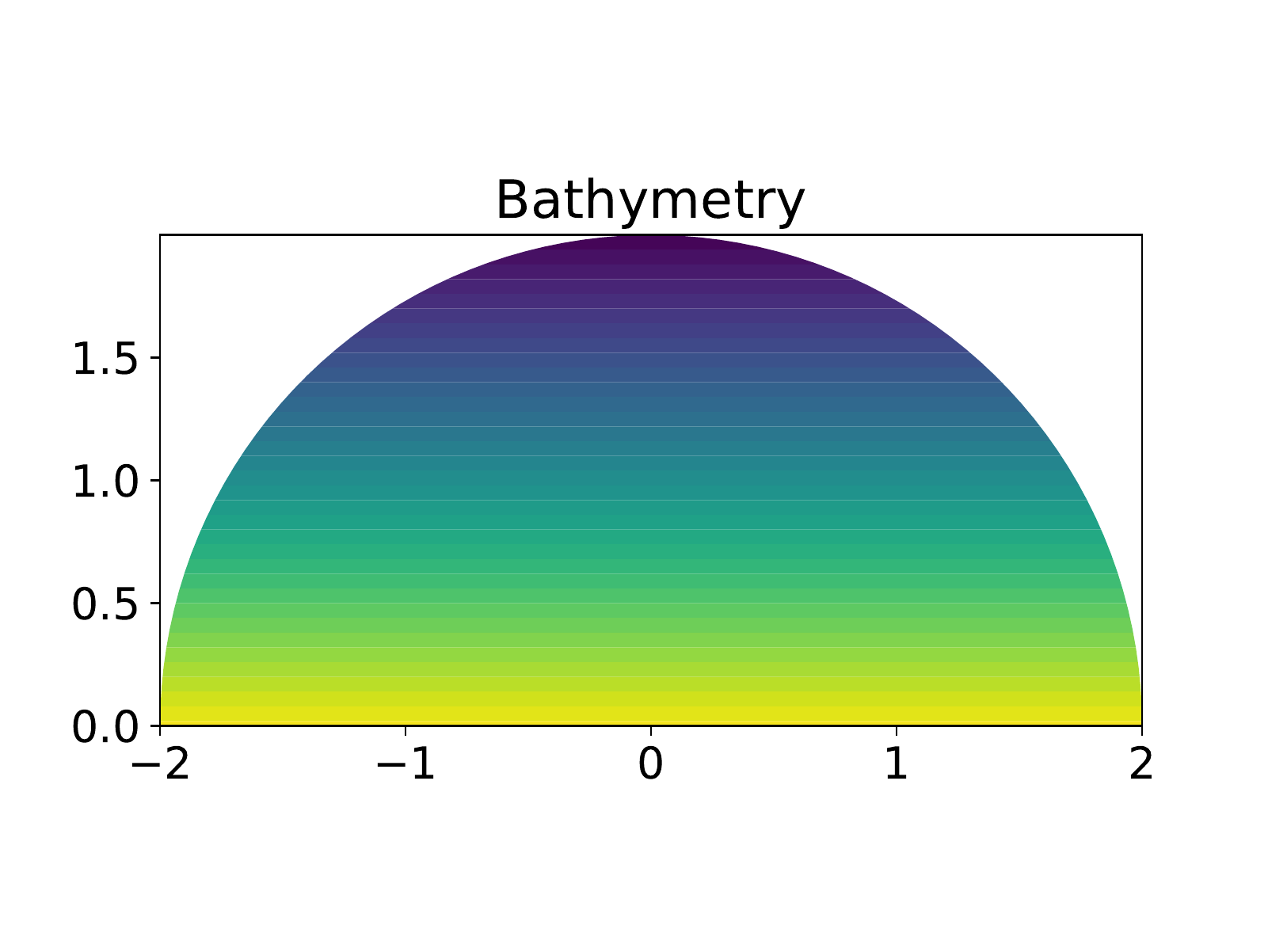}};
	\node[above left = -1cm and -0.01cm of 4](5)
	{\tiny 2.};
	\node[above left = -0.87cm and -0.01cm of 3](6)
	{\tiny 3.};
	\end{tikzpicture}
	\caption[Ex.1 Mesh, Water Height, Velocities \& Sediment]{1.: Initial Mesh and Porosity with Enlarged Image Section, 2.: Linear Bathymetry, 3.: Field State at $t=0.1$}
	\label{fig:InitHalfCircle}
\end{figure}
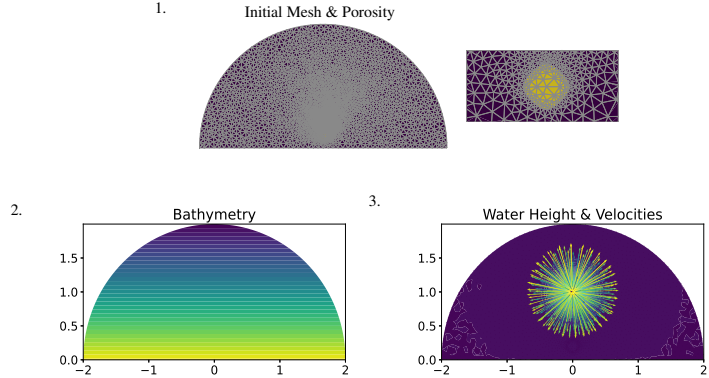

The material coefficient is at $\phi_2=0.4$ at $D$ and we obtain classical SWE on $\tilde{\Om}$ by setting $\phi_1=1$. In addition, we employ Gaussian initial conditions as $(H_0+z,uH_0,vH_0)=(1+\exp(-15x^2 - 15(y-1)^2),0,0)$, which result into a wave travelling in time towards the boundaries. We prescribe the boundary conditions as before, using rigid-wall and outflow boundary conditions for $\Ga$ and $\Gd$.
In this example we have used a backward Euler time-scheme, that arises from the theta-method for $\theta=1$, such as a SIP-DG-method of first order that was described before in Section \ref{Sec:SIPGDG}. For the spatial discretization, we have used $C_{IP}=20$ in the SIP-DG method. Solving the state equations requires the definition of the time-horizon $T=2$, which is chosen to include the travel of a wave to and from the shore using a time-stepping size of $\diff t=\expnumber{2}{-3}$. Our calculations are performed for a linear decreasing time-constant sediment $z=0.5-0.25y$.
Having solved state and adjoint equations the mesh deformation is performed for initial step size $\rho=1.5$ as described in Section \ref{Sec:Implementation}, where we specify $\mu_{min}=10$ and $\mu_{max}=100$ in (\ref{Eq:33Lame}). In Figure \ref{fig:OptiMesh} the result of the shape optimization procedure is displayed after $24$ iterations, where deformations appear to be symmetric. 
\begin{figure}[!htbp]
	\centering
	\begin{tikzpicture}
	\node[anchor=south west,inner sep=0] (1) {\vspace{-0.9cm}
	\includegraphics[scale=0.13]{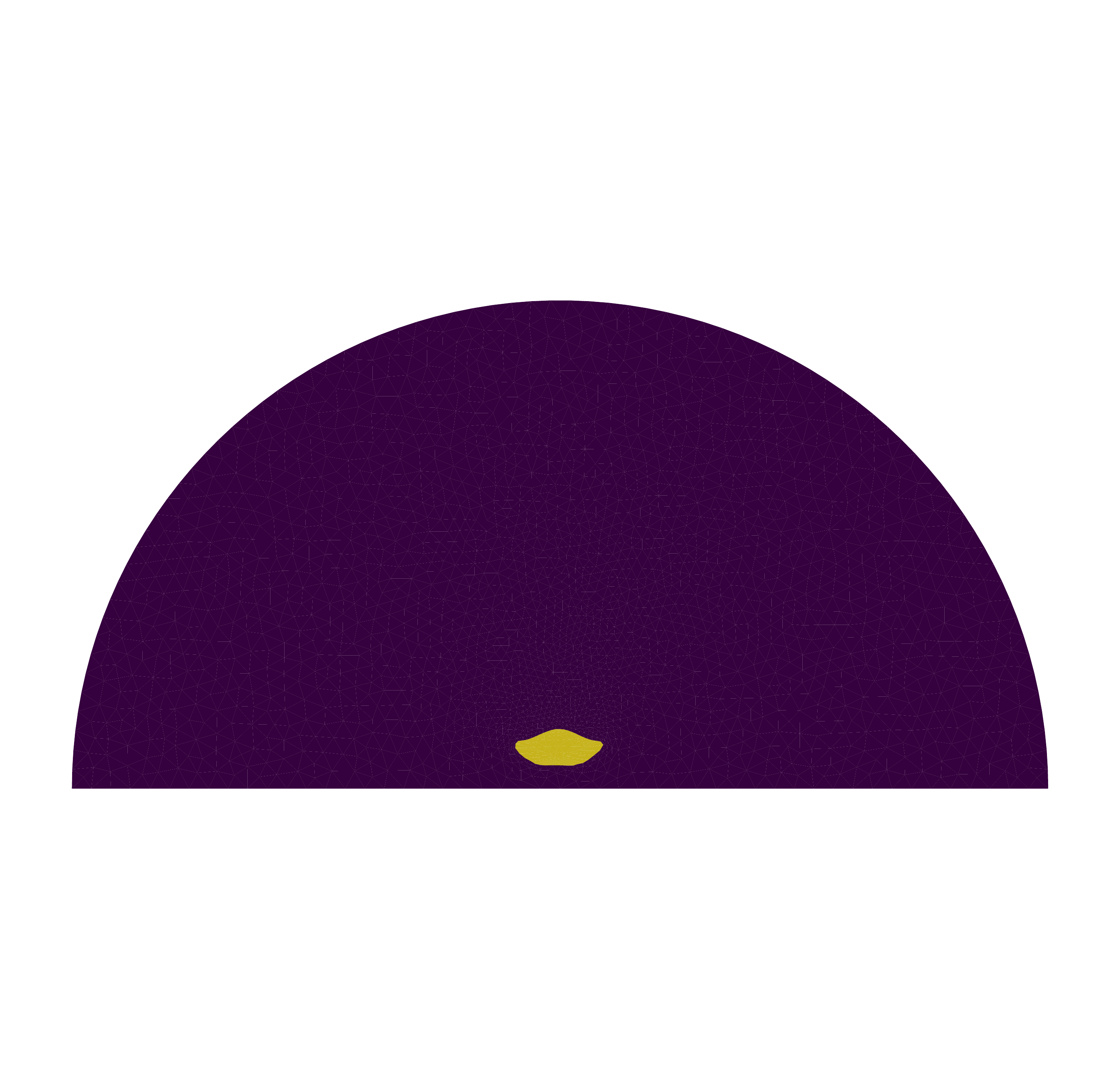}};
	\node[above= -1.1cm of 1] (7) 
	{\tiny Optimized Porosity};
	\node[right= -0.1cm of 1] (7) 
	{\includegraphics[scale=0.07]{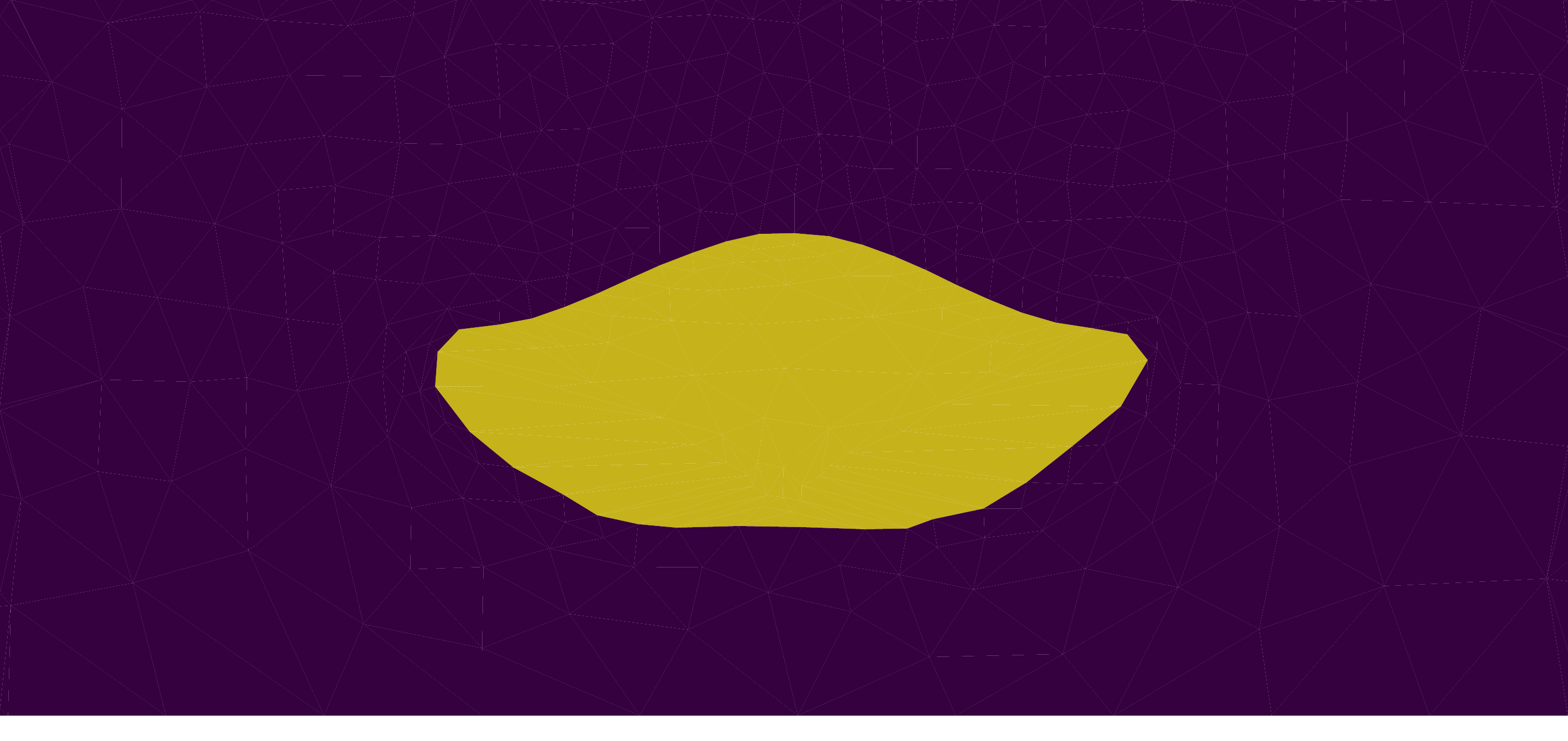}};
	\end{tikzpicture}
	\caption[Ex. Optimized Shape \& Close up]{1.:  Optimized Porous Region with Enlarged Image Section}
	\label{fig:OptiMesh}
\end{figure}
As we observe in Figure \ref{fig:OptiHalfCircle}, we have achieved a notable decrease in the objective functional.
\begin{figure}[!htbp]
	\centering
	\begin{tikzpicture}
	\node[anchor=south west,inner sep=0] (1) 
	{\scalebox{0.48}{\input{objective_porous_3.tex}}};
	\end{tikzpicture}
	\caption[Ex. Target Functional]{Objective Value per Iteration}
	\label{fig:OptiHalfCircle}
\end{figure}
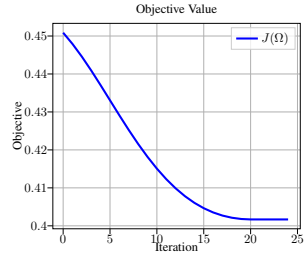

\subsection{Example: The Mentawai Islands}\label{sec:mentawai}
The second example will investigate an archipelago in the southwest of Sumatra, Indonesia the Mentawai islands, which have turned out to be an effective shield in the 2004 and 2010 tsunami for the mainland located behind \cite{Stefanakis2014}. Mentawai islands are threatened by rising sea levels and victim to massive floodings in the last decades and are hence offering itself for protective measures.
Real coastal applications require suitable mesh representations. Shorelines are taken from the GSHHG\footnote[1]{https://www.ngdc.noaa.gov/mgg/shorelines/  (last visited May 5, 2022)} databank, where we use a geographical information system QGIS3 to process the data to GMSH for the mesh generation \cite{Avdis2016}. For computational ease, we have decided to not consider smaller islands of a diameter less than $5$km.  Similar to the preceding example, we interpret $\Ga$ as coastline of the mainland, $\Gd$ as the open sea boundary such as  $\Ge$ as the interface boundary of the offshore islands. Assuming that islands are flooded, we represent them by a difference in the material coefficient, which shape is to be optimized. For this we set $\phi_2=0.5$ at $D$ and $\phi_1=1$ on $\tilde{\Om}$ (cf. to Figure \ref{fig:InitHalfCircle2}, 1. Subfigure). As before, we are in a tsunami-like setting and start with suitable Gaussian initial conditions for the height of the water. For simplicity, the sediment height is assumed to be zero on the whole domain. The remaining model-settings are similar to Section \ref{sec:exhalfcircled}.
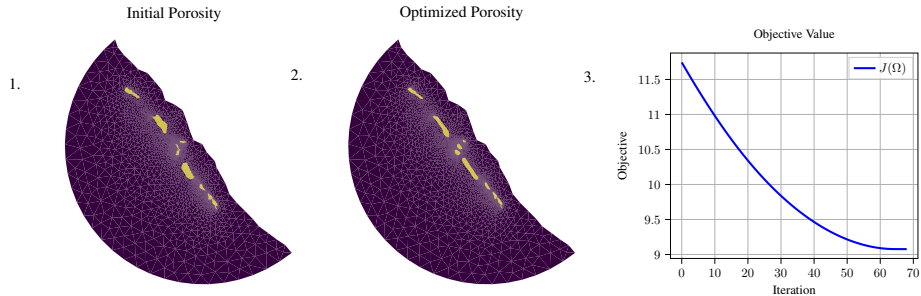
\begin{figure}[htb!]
	\centering
	\begin{tikzpicture}
	\node[anchor=south west,inner sep=0] (1) {\vspace{-0.9cm}
	\includegraphics[scale=0.13]{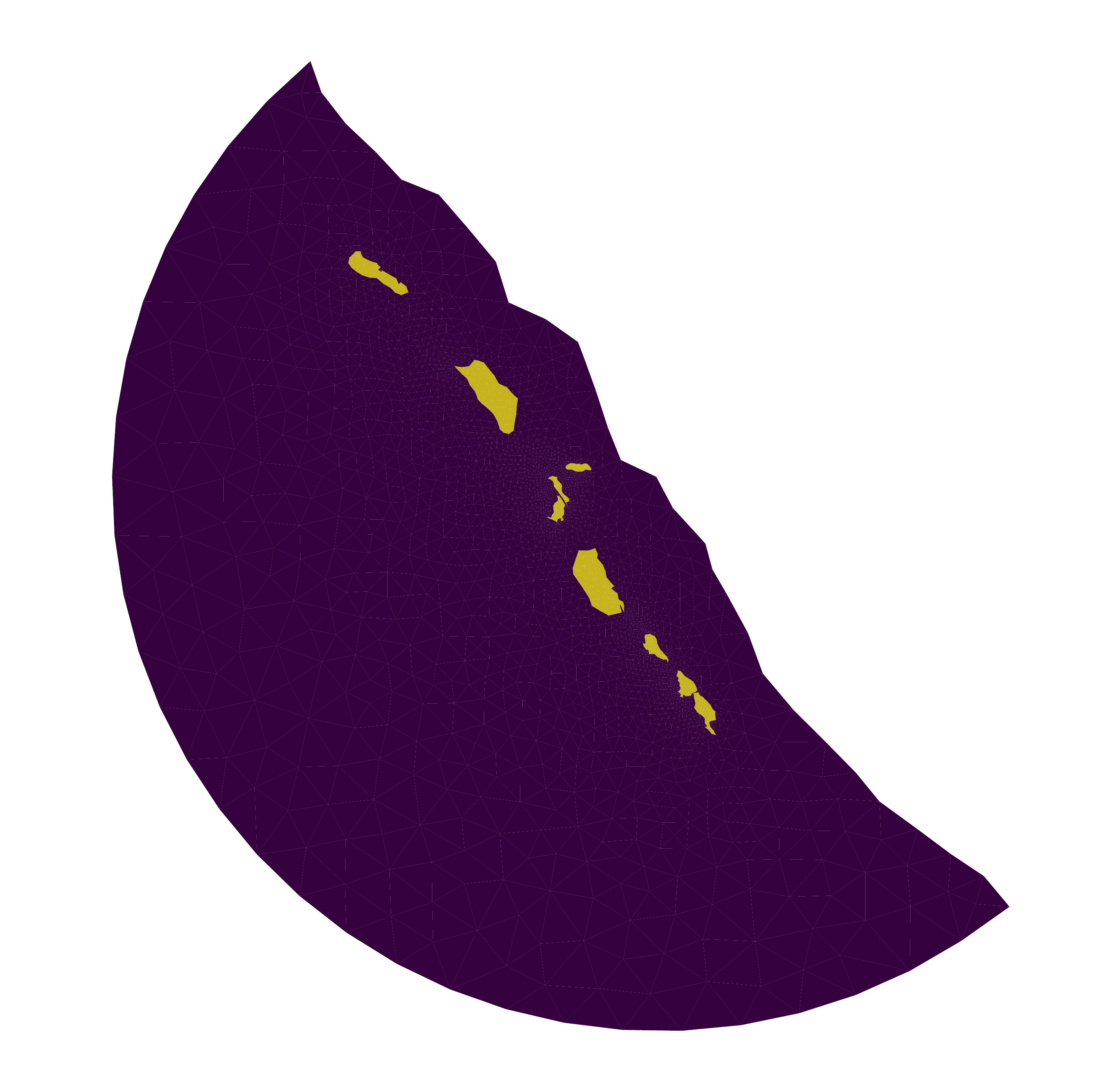}};
	\node[above= -0.1cm of 1] (7) 
	{\tiny Initial Porosity};
	\node[right= -0.1cm of 1] (2) 
	{\includegraphics[scale=0.13]{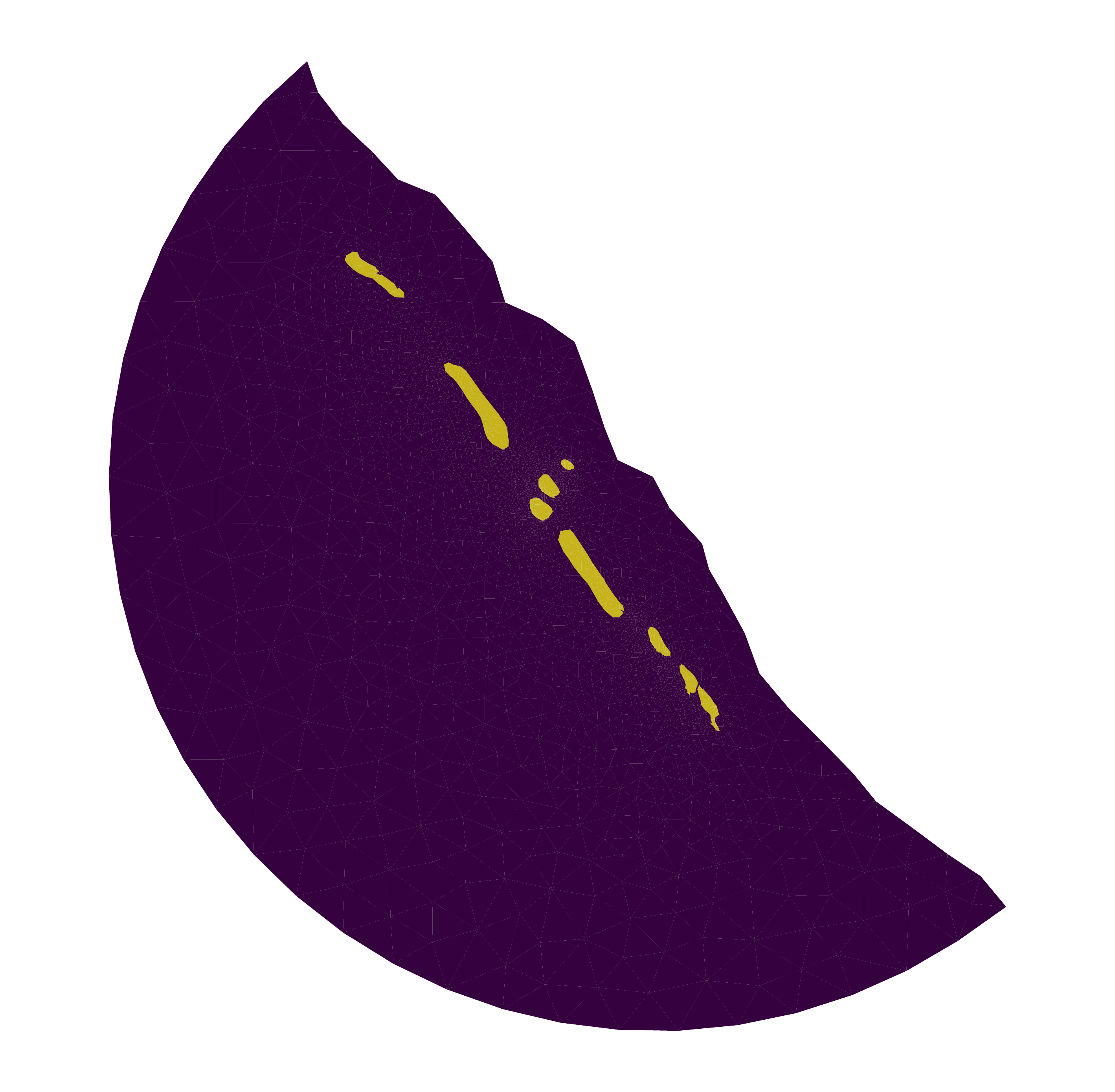}};
	\node[above= -0.2cm of 2] (8) 
	{\tiny Optimized Porosity};
	\node[above left = -1cm and -0.01cm of 1](4)
	{\tiny 1.};
	\node[right = -0.1cm of 2](3)
	{\scalebox{0.48}{\input{objective_porous_mentawai.tex}}};
	\node[above left = -1cm and -0.01cm of 2](5)
	{\tiny 2.};
	\node[above left = -1cm and -0.01cm of 3](6)
	{\tiny 3.};
	\end{tikzpicture}
	\caption[Ex.1 Mesh, Water Height, Velocities \& Sediment]{1.: Initial Porosity, 2.: Optimized Porosity, 3.: Objective Value per Iteration}
	\label{fig:InitHalfCircle2}
\end{figure}
We can once more observe convergence in the objective, after applying shape optimization on the porous region  (cf. to Figure \ref{fig:InitHalfCircle2}, 2. \& 3. Subfigure). The obstacles are enlarged in perpendicular direction to the incoming sea wave.
\FloatBarrier

\section{Conclusion}
We have investigated porous Shallow Water Equations, where the difference in the material coefficient can be interpreted as a permeable obstacle, that is placed in before shores. We have derived a well-balanced SIP-DG scheme to solve for the evolution of the waves. Based on this solution we have derived the time-dependent continuous adjoint and shape derivative in volume form. Results were tested successfully in two scenarios, where the obstacle's shape has been optimized to target a rest height at the shore. Results can be easily adjusted for arbitrary meshes, objective functions and different wave properties driven by initial and boundary conditions as partially shown in \cite{Schlegel20212}.

\section*{Acknowledgement}
This work has been supported by the Deutsche
Forschungsgemeinschaft within the Priority program SPP 1962 "Non-smooth and Complementarity-based Distributed Parameter Systems: Simulation and Hierarchical Optimization". The authors would like to thank Diaraf Seck (Université Cheikh Anta Diop, Dakar, Senegal) and Mame Gor Ngom (Université Cheikh Anta Diop, Dakar, Senegal) for helpful and interesting discussions within the project Shape Optimization Mitigating Coastal Erosion (SOMICE). 

\bibliographystyle{unsrt}
\bibliography{bibliography}  %%% Uncomment this line and comment out the ``thebibliography'' section below to use the external .bib file (using bibtex) .
\appendix
\section{Derivation of Viscous Porous Momentum}
\label{app:ViscousPorMom}
We hereby follow \cite{Soares2006} and extend for diffusive terms. In addition, we assume a water density $\rho=1$ and remark that the volume $V$ of water in a control volume is given by 
\begin{align}
V=\int_{y_0}^{y_0+\delta y}\int_{x_0}^{x_0+\delta x}\phi H\diff x \diff y
\end{align}
with $x_0, y_0$ being the coordinates of the lower left corner of the control volume.
Following \cite{Soares2006} the momentum balance in $x$-direction can be written as
\begin{align}
M=\frac{\partial M_x}{\partial t}-F_{M,W}+F_{M,E}-F_{M,S}+F_{M,N}-P_W+P_E-W_x-B_x-R_x=0\text{,}
\end{align}
where $F$-terms account for $x$-Momentum fluxes, $P$-terms for pressure forces, such as $W,B$ and $R$-terms for porosity influence such as bottom pressure and friction terms with indices representing the western $W$, eastern $E$ northern $N$ and southern $S$ sides of the control volume. For viscous SWE as in \cite{Agoshkov1994} this momentum balance is extended by the volume diffusion through the western and eastern side as
\begin{align}
D_{W}&=-\int_{y_0}^{y_0+\delta y}\mu\phi  H u_x(x_0,y)\diff y\\
D_{E}&=-\int_{y_0}^{y_0+\delta y}\mu\phi H u_x(x_0+\delta x, y)\diff y 
\end{align}
such as southern and northern sides
\begin{align}
D_{S}&=-\int_{x_0}^{x_0+\delta x}\mu\phi H u_y(x,y_0)\diff x\\
D_{N}&=-\int_{x_0}^{x_0+\delta x}\mu\phi H u_y(x,y_0+\delta y)\diff x \text{,}
\end{align}
where $u_x$ and $u_y$ denotes the first spatial derivatives with respect to $x$ and $y$ and diffusion coefficient $\mu$. Momentum balancing these terms then leads to
\begin{align}
M-D_{W}+D_{E}-D_{S}+D_{N}=0
\label{Mombalance}
\end{align}
Substituting terms in (\ref{Mombalance})
gives 
\begin{equation}
\begin{aligned}
M+&\int_{y_0}^{y_0+\delta y}\mu\phi H u_x(x_0,y)\diff y-\int_{y_0}^{y_0+\delta y}\mu\phi H u_x(x_0+\delta x, y)\diff y+\\
&\int_{x_0}^{x_0+\delta x}\mu\phi H u_y(x,y_0)\diff x-\int_{x_0}^{x_0+\delta x}\mu\phi H u_y(x,y_0+\delta y)\diff x=0\text{.}
\label{Eq:MombalanceInt}
\end{aligned}
\end{equation}
When $\delta x$ and analogously $\delta y$ tend to 0, it holds
\begin{align}
\lim_{\delta x\rightarrow 0}(\mu\phi H u_x)(x_0+\delta x, y)-(\mu\phi H u_x)(x_0, y)=\delta x \frac{\partial}{\partial x}(\mu\phi Hu_x)
\end{align}
such that an evaluation of integrals in (\ref{Eq:MombalanceInt}) yields
\begin{align}
\delta x \delta y M-\delta x \delta y \frac{\partial}{\partial x}(\mu\phi Hu_x)-\delta x \delta y \frac{\partial}{\partial y}(\mu\phi Hu_y)=0\text{,}
\end{align}
which gives the momentum balance as
\begin{align}
M-\frac{\partial}{\partial x}(\mu\phi Hu_x)- \frac{\partial}{\partial y}(\mu\phi Hu_y)=0\text{.}
\end{align}
The y-momentum is derived in accordance.

\newpage
\section{Derivation of Adjoint Equations}\label{app:derivadj}
\begin{proof}
	We perform integration by parts once more on time and spatial derivatives of the weak form (\ref{Eq:17aweak}), where boundaries are denoted as in Section \ref{sec:PF}, to obtain
	\allowdisplaybreaks
	\begin{equation}
	\begin{aligned}
	a(H,\Qv,p,\rv)=&\int_0^T\int_\Om-\frac{\partial p}{\partial t}\phi 
	H\diff x\diff t+\int_\Om \phi\left[H(T,x)p(T,x)-H_0p(0,x)\right]\diff x+\\
	%--------------------------------- flux sediment
	&\int_0^T\int_\Om-\phi\Qv\cdot\nab p\diff x\diff t+\int_0^T\int_{\Gamma_{out}} p\phi\Qv\cdot \nv\diff s\diff t+\\
	%--------------------------------- flux boundary interface
	&\int_0^T\int_{\Ge} [\![p\phi\Qv\cdot \nv]\!]\diff s\diff t+\\
	%--------------------------------- flux sediment
	&\int_0^T\int_\Om-(H+z)\nab\cdot(\mu_v\nab p)\diff x\diff t+\\
	%--------------------------------- viscous sediment
	&\int_0^T\int_{\Gamma_{out}}\left[\mu_v(H+z)\nab p\cdot \nv-p\mu_v\nab (H+z)\cdot \nv\right]\diff s\diff t+\\
	%--------------------------------- viscous sediment interface
	&\int_0^T\int_{\Ge}\left[[\![\mu_v\nab(H+z)\cdot\nv p]\!]-[\![\mu_v(H+z)\nab p\cdot\nv]\!]\right]\diff s\diff t+\\
	%--------------------------------- time momentum
	&\int_0^T\int_\Om-\frac{\partial \rv}{\partial t}\cdot(\phi\Qv)\diff x\diff t+\int_\Om \phi\left[\Qv(T,x)\cdot\rv(T,x)-\Qv_0\cdot\rv(0,x)\right]\diff x+\\
	%--------------------------------- 1. flux momentum source terms
	&\int_0^T\int_\Om-\phi\frac{\Qv}{H}\cdot\nab\rv\cdot\Qv\diff x\diff t+\int_0^T\int_{\Gamma_{out}} \phi\frac{\Qv}{H}\cdot\rv\Qv\cdot \nv\diff s\diff t+\\
	%--------------------------------- 2. flux momentum
	&\int_0^T\int_\Om-\frac{1}{2}g\phi H^2\nab\cdot\rv\diff x\diff t+\int_0^T\int_{\Gamma_{out}} \frac{1}{2}g\phi H^2\rv\cdot \nv\diff s\diff t+\\
	%--------------------------------- 2. flux momentum interface boundary
	&\int_0^T\int_{\Ge} [\![\phi\frac{\Qv}{H}\cdot\rv\Qv\cdot \nv]\!]\diff s\diff t+\int_0^T\int_{\Ge} [\![\frac{1}{2}g\phi H^2\rv\cdot \nv]\!]\diff s\diff t+\\
	%--------------------------------- 2. flux momentum
	&\int_0^T\int_\Om-\Qv\cdot
	\nab\cdot(\mu_f\phi\nab \rv)\diff x\diff t+\int_0^T\int_\Om g\phi H\nab z \cdot\rv\diff x\diff t+\\
	%--------------------------------- viscous momentum
	&\int_0^T\int_\Gamma\left[\mu_f\phi\Qv\cdot\nab \rv\cdot \nv-\rv\cdot (\mu_f\phi \nab \Qv)\cdot \nv\right]\diff s\diff t+\\
	%--------------------------------- 
	&\int_0^T\int_{\Ge}\left[[\![\mu_f\phi \nab(\Qv)\cdot\nv\cdot\rv]\!]-[\![\mu_f\phi \nab\rv\cdot\Qv\cdot\nv]\!]\right]\diff s\diff t-\\
	%--------------------------------- 2. porous source
	&\int_0^T\int_\Om g\frac{1}{2}H^2\nab\phi\cdot\rv\diff s\diff t\text{.}
	\end{aligned}
	\label{App:Step1Adjoint}
	\end{equation}
	\allowdisplaybreaks
	Using the jump identity $[\![ab]\!]=\{\!\{a\}\!\}[\![b]\!]+\{\!\{b\}\!\}[\![a]\!]$ on boundary integrals over the interface $\Ge$ and inserting Boundary Conditions (\ref{Eq:BCSWE}) on $\Ga$ and $\Gd$ for terms that arise from the diffusive fluxes lead to
	\allowdisplaybreaks
	\begin{equation}
	\allowdisplaybreaks
	\begin{aligned}
	\allowdisplaybreaks
	a(H,\Qv,p,\rv)=&\int_0^T\int_\Om-\frac{\partial p}{\partial t}\phi H\diff x\diff t+\int_\Om \phi\left[H(T,x)p(T,x)-H_0p(0,x)\right]\diff x-\\
	%--------------------------------- flux continuity
	&\int_0^T\int_\Om\phi\Qv\cdot\nab p\diff x\diff t+\int_{0}^T\int_{\Gd}\phi p\Qv\cdot\nv\diff s\diff t+	\\
	%---------------------------------
	&\int_0^T\int_{\Ge} [\![p\phi\Qv\cdot \nv]\!]\diff s\diff t-\int_0^T\int_\Om\frac{1}{2}g\phi H^2\nab\cdot\rv\diff x\diff t-\\
	%--------------------------------- viscous continuity
	&\int_0^T\int_\Om(H+z)\nab\cdot(\mu_v\nab p)\diff x\diff t+\int_0^T\int_{\Ga}\mu_v(H+z)\nab p\cdot \nv\diff s\diff t+\\
	&\int_0^T\int_{\Gd}\mu_vH_1\nab p\cdot \nv\diff s\diff t+\int_0^T\int_{\Gd}-p\mu_v\nab (H_1+z)\cdot \nv\diff s\diff t+\\
	%--------------------------------- viscous continuity
	&\int_0^T\int_{\Ge}\left[\{\!\{\mu_v\nab(H+z)\cdot\nv\}\!\}[\![p]\!]-\{\!\{H+z\}\!\}[\![\mu_v\nab p\cdot\nv]\!]\right]\diff s\diff t-\\
	%--------------------------------- time momentum
	&\int_0^T\int_\Om\frac{\partial \rv}{\partial t}\cdot(\phi\Qv)\diff x\diff t+\int_\Om\phi \left[\Qv(T,x)\cdot\rv(T,x)-\Qv_0\cdot\rv(0,x)\right]\diff x-\\
	%--------------------------------- 1. flux momentum
	&\int_0^T\int_\Om\phi\frac{\Qv}{H}\cdot\nab\rv\cdot\Qv\diff x\diff t+\int_0^T\int_{\Gd} \phi\frac{\Qv}{H_1}\cdot\rv\Qv\cdot \nv\diff s\diff t+\\
	%--------------------------------- 2. flux momentum
	&\int_0^T\int_{\Ga} \frac{1}{2}g\phi H^2\rv\cdot \nv\diff s\diff t+\int_0^T\int_{\Gd} \frac{1}{2}g\phi H_1^2\rv\cdot \nv\diff s\diff t+\\
	%--------------------------------- 2. flux momentum interface boundary
	&\int_0^T\int_{\Ge} [\![\phi\frac{\Qv}{H}\cdot\rv\Qv\cdot \nv]\!]\diff s\diff t+\int_0^T\int_{\Ge} [\![\frac{1}{2}g\phi H^2\rv\cdot \nv]\!]\diff s\diff t-\\
	%--------------------------------- viscous momentum
	&\int_0^T\int_\Om (\Qv)\cdot
	\nab\cdot(\mu_f\phi\nab \rv)\diff x\diff t+\int_0^T\int_{\Ga,\Gd}\mu_f\phi \Qv\nab \rv\cdot\nv\diff s\diff t+\\
	%--------------------------------- viscous momentum interface
	&\int_0^T\int_{\Ge}\left[\{\!\{\mu_f\phi \nab(\Qv)\cdot\nv\}\!\}\cdot[\![\rv]\!]-\{\!\{\Qv\}\!\}\cdot[\![\mu_f\phi \nab\rv\cdot\nv]\!]\right]\diff s\diff t+\\
	%--------------------------------- source terms and additional for spacing
	&\int_0^T\int_\Om g\phi H\nab z\cdot\rv\diff x\diff t-\int_0^T\int_\Om g\frac{1}{2}H^2\nab\phi\cdot\rv\diff x\diff t\text{.}
	\end{aligned}
	\end{equation}
	Differentiating for the state variable $H$ leads to
	\begin{equation}
	\begin{aligned}
	\frac{\partial a(H,\Qv,p,\rv)}{\partial H}=&\int_0^T\int_\Om-\frac{\partial \phi p}{\partial t}\diff x\diff t+\int_\Om \phi p(T,x)\diff x+\\
	%--------------------------------- viscous sediment
	&\int_0^T\int_\Om-\nab\cdot(\mu_v\nab p)\diff x\diff t+\int_0^T\int_{\Ga}\left[\mu_v\nab p\cdot \nv\right]\diff s\diff t+\\
	%--------------------------------- viscous interface
	&\int_0^T\int_{\Ge}\left[\{\!\{\mu_v\nab(H+z)\cdot\nv\}\!\}_H[\![p]\!]-\{\!\{H+z\}\!\}_H[\![\mu_v\nab p\cdot\nv]\!]\right]\diff s\diff t-\\
	%--------------------------------- 1. flux momentum
	&\int_0^T\int_\Om\phi\frac{\Qv}{H^2}\cdot\nab\rv\cdot\Qv\diff x\diff t-\\
	%--------------------------------- 2. flux momentum interface boundary
	&\int_0^T\int_{\Ge} [\![\phi\frac{\Qv}{H^2}\cdot\rv\Qv\cdot \nv]\!]\diff s\diff t+\int_0^T\int_{\Ge} [\![g\phi H\rv\cdot \nv]\!]\diff s\diff t-\\
	%--------------------------------- 2. flux momentum
	&\int_0^T\int_\Om-g\phi H\nab\cdot\rv\diff x\diff t+\int_0^T\int_{\Ga} g\phi H\rv\cdot \nv\diff s\diff t+\\
	%--------------------------------- source terms
	&\int_0^T\int_\Om g\phi\nab z\cdot\rv\diff x\diff t-\int_0^T\int_\Om gH\nab\phi\cdot\rv\diff x\diff t
	\label{adjderh}
	\end{aligned}
	\end{equation}
	and w.r.t. $\Qv$ to
	\begin{equation}
	\begin{aligned}
	%--------------------------------- time momentum
	\frac{\partial a(H,\Qv,p,\rv)}{\partial \Qv}=		&\int_0^T\int_\Om-\frac{\partial \phi\rv}{\partial t}\diff x\diff t+\int_\Om \phi\rv(T,x)\diff x-\\
	%--------------------------------- flux sediment
	&\int_0^T\int_\Om\phi\nab p\diff x\diff t+\int_0^T\int_{\Gd}\phi p\nv\diff s\diff t+\int_0^T\int_{\Ge} [\![p\phi\nv]\!]\\
	%--------------------------------- 1. flux momentum
	&\int_0^T\int_\Om-\phi\frac{1}{H}(\nab\rv)^T\Qv-\frac{1}{H}(\Qv\cdot\nab)\rv\Qv\diff x\diff t+\\
	%--------------------------------- 1. flux momentum boundary terms
	&\int_0^T\int_{\Gd} \frac{\phi}{H_1}(\Qv\cdot \nv)\rv\diff s\diff t+\int_0^T\int_{\Gd} \frac{\phi}{H_1}(\Qv\rv) \cdot \nv\diff s\diff t+\\
	%--------------------------------- 2. flux momentum interface boundary
	&\int_0^T\int_{\Ge} [\![\frac{\phi}{H}(\Qv\cdot \nv)\rv]\!]\diff s\diff t+\int_0^T\int_{\Ge} [\![ \frac{\phi}{H}(\Qv\rv) \cdot \nv]\!]\diff s\diff t-\\
	%--------------------------------- viscous momentum
	&\int_0^T\int_\Om\nab\cdot(\mu_f\phi \nab\rv)\diff x\diff t+\int_0^T\int_{\Ga,\Gd}\mu_f\phi \nab\rv\nv\diff s\diff t+\\
	%--------------------------------- viscous interface
	&\int_0^T\int_{\Ge}\left[\{\!\{\mu_f\phi \nab(\Qv)\nv\}\!\}_{\Qv} [\![\rv]\!]-\{\!\{\Qv\}\!\}_{\Qv}[\![\mu_f\phi \nab \rv\nv]\!]\right]\diff s\diff t\text{,}
	\label{adjderq}
	\end{aligned}
	\end{equation}
where the subscript denotes differentiation for the respective state variable.	
Now if $\frac{\partial a(H,\Qv,p,\rv)}{\partial U}=-\frac{\partial J_1}{\partial U}$ then $\frac{\partial\mathcal{L}}{\partial U}=0$ is fulfilled. From this we get the  adjoint equations in strong form (\ref{Eq:19Adjoint1}) and (\ref{Eq:19Adjoint2}) with boundary conditions from equating boundary terms to zero in (\ref{adjderh}) and (\ref{adjderq}). 	
\end{proof}

\section{Derivation of Shape Derivative}\label{app:derivsha}
\begin{proof}
	We regard the Lagrangian (\ref{Eq:16Lag}). As in \cite{Schulz2014b}, the theorem of Correa and Seger \cite{Correa1985} is applied on the right hand side of 
	\begin{align}
	J_1(\Om)=\min_{U}\max_{P}\Lag(\Om,U,P)
	\label{Eq:191Saddle}
	\end{align}
	The assumptions of this theorem can be verified as in \cite{Delfour2011}. We now use the definition of the shape derivative (\ref{Eq:7EulerDer}) in terms of the Lagrangian, i.e.
	\allowdisplaybreaks
	\begin{align*}
	D&\mathcal{L}(\Om,U,P)[\Vv]\\
	%---------------------------------------------
	&=\lim_{\eps\rightarrow 0^+}\frac{\mathcal{L}(\Om_\eps;U,P)-\mathcal{L}(\Om;U,P)}{\eps}\\
	%---------------------------------------------
	&=\frac{d}{d\eps}\restr{\mathcal{L}(\Om_\eps,U,P)}{\eps=0^+}=\frac{d}{d\eps}\restr{\mathcal{L}(\Om_\eps,H,\Qv,p,\rv)}{\eps=0^+}
	\end{align*}
	%--------------------------------------------- Insert Material Derivative 1.1. td H
	and apply the rule for differentiating domain integrals (\ref{Eq:13DoaminR}), where we split integrals for readability in to be added domain part, i.e.
\begin{align*}		
	&\int_\Om\Big[\int_0^T-D_m\left(\frac{\partial p}{\partial t}\phi H\right)\diff t+ D_m\left(\phi H(T,x)p(T,x)-\phi H_0p(0,x)\right)+\\
	%--------------------------------------------- 1.1 td Q 
	&\quad\int_0^T-D_m\left(\frac{\partial \rv}{\partial t}\cdot\phi\Qv\right)\diff t+ D_m\left(\phi\Qv(T,x)\cdot\rv(T,x)-\phi\Qv_0\cdot\rv(0,x)\right)+\\
	%--------------------------------------------- 
	&\quad\int_0^T D_m\left(\nab\cdot(\phi\Qv) p\right)\diff t+\int_0^TD_m\left(\mu_v\nab (H+z)\cdot\nab p\right)\diff t+\\
	%---------------------------------------------
	&\quad\int_{0}^TD_m\left(\nab\cdot\left(\phi\frac{\Qv}{H}\otimes \Qv\right)\cdot\rv\right)\diff t+
	\int_0^T+D_m\left(\frac{1}{2}g\nab (\phi H^2)\cdot\rv\right)\diff t+\\
	%---------------------------------------------
	&\quad\int_0^TD_m\left(\phi\mu_f \nab \Qv:\nab\rv\right)\diff t+\int_0^TD_m\left(gH\nab z\cdot\rv\right)\diff t+\\
	%---------------------------------------------
	&\quad\int_0^T-D_m\left(\frac{1}{2}gH^2\nab \phi\cdot\rv\right)\diff t+
	\\
	%--------------------------------------------- div terms
	&\quad div(\Vv)\Big(\int_0^T-\frac{\partial p}{\partial t}\phi H\diff t+\phi H(T,x)p(T,x)-\phi H_0p(0,x)+\\
	%--------------------------------------------- 
	&\quad\int_0^T-\frac{\partial \rv}{\partial t}\cdot(\phi\Qv)\diff t+ \phi\Qv(T,x)\cdot\rv(T,x)-\phi\Qv_0\cdot\rv(0,x)+\int_0^T \nab\Qv\cdot p\diff t+\\
	%---------------------------------------------
	&\quad\int_0^T\mu_v\nab (H+z)\cdot\nab p\diff t+\quad\int_{0}^T\nab\cdot\left(\phi\frac{\Qv}{H}\otimes \Qv\right)\cdot\rv\diff t+\\
	%---------------------------------------------
	&\quad\int_0^T+\frac{1}{2}g\nab(\phi H^2)\cdot\rv\diff t+\int_0^T\phi\mu_f \nab \Qv:\nab\rv\diff t+\int_0^Tg\phi H\nab z\cdot\rv\diff t-\\
	%---------------------------------------------
	&\quad\int_0^Tg H^2\nab\phi\cdot\rv\diff t
	\Big)\Big]\diff x
	\end{align*}
	and interior such as exterior boundary part, i.e.
	\begin{align*}
	%--------------------------------------------- Boundary Terms
	&\int_{\Ga}\Big[\frac{1}{2}\int_0^TD_m\left([N(\hat{U}(t,x)-\bar{U}(t,x))]^2\right)+div_{\Ga}(\Vv)[N(\hat{U}(t,x)-\bar{U}(t,x))]^2\diff t\Big]\diff s-\\
	%-------------------------------------------- 
	&\quad\int_{\Gd}\Big[\int_0^TD_m\left(\mu_v\nab (H_1+z) \cdot \nv p\diff t\right)+ div_{\Gd}(\Vv)\Big(\int_0^T\mu_v\nab (H_1+z)\cdot \nv p\diff t\Big)\Big]\diff s-\\
	%---------------------------------------------
	&\quad\int_{\Ge}\Big[\int_0^TD_m\left( [\![\mu_v\nab (H+z) \cdot \nv p]\!]\right)\diff t+\int_0^TD_m\left([\![\phi\mu_f\nab\Qv \cdot \nv\cdot\rv]\!]\right)\diff t+\\
	&\quad div_{\Ge}(\Vv)\Big(\int_0^T[\![-\mu_v\nab(H+z)\cdot \nv p]\!]\diff t+\int_0^T[\![-\phi\mu_f\nab \Qv\cdot \nv\cdot\rv]\!]\diff t\Big)\Big]\diff s\text{,}
	\end{align*}
	where $div_{\Gamma}(\Vv)=div(\Vv)-\nv\cdot(\nab\Vv)\nv$ is the tangential divergence of the vector field $\Vv$ for the respective boundary normal $\nv$.
	Now the product rule (\ref{Eq:10MatProdR}) yields respectively for the domain part
	\allowdisplaybreaks
	\begin{align*}
	%--------------------------------------------- Insert Product Rule and so on
	\quad\quad\quad=&\int_\Om\Big[\int_0^T-D_m\left(\frac{\partial p}{\partial t}\right)\phi H-\frac{\partial p}{\partial t}D_m(\phi H)\diff t+ \\
	%---------------------------------------------
	&D_m(\phi H(T,x))p(T,x)+H(T,x)\dot{p}(T,x)-\phi H_0\dot{p}(0,x)+\\
	%--------------------------------------------- 1.1 td Q 
	&\int_0^T-D_m\left(\frac{\partial \rv}{\partial t}\right)\cdot(\phi\Qv)-\frac{\partial \rv}{\partial t}\cdot D_m(\phi\Qv)\diff t+ D_m(\Qv(T,x))\cdot\rv(T,x)+\\
	%--------------------------------------------- 
	&\phi\Qv(T,x)\cdot\dot{\rv}(T,x)-\phi\Qv_0\cdot\dot{\rv}(0,x)+
	\int_0^T \dot{p}\cdot\nab (\phi\Qv)+p\cdot D_m(\nab (\phi\Qv))\diff t+\\
	%--------------------------------------------- 
	&\int_0^T\left(\mu_vD_m(\nab (H+z))\cdot\nab p+\mu_v\nab (H+z)\cdot D_m(\nab p)\right)\diff t-\\
	%---------------------------------------------
	&\int_{0}^TD_m\left(\nab\cdot\left(\phi\frac{\Qv}{H}\otimes \Qv\right)\right)\cdot\rv\diff t+\int_{0}^T\nab\cdot\left(\phi\frac{\Qv}{H}\otimes \Qv\right)\cdot D_m\left(\rv\right)\diff t+\\
	%---------------------------------------------
	& \int_0^T\left(\frac{1}{2}gD_m(\nab (\phi H^2))\cdot\rv+\frac{1}{2}g\nab (\phi H^2)\cdot D_m(\rv)\right)\diff t+
	\\
	%---------------------------------------------
	&\int_0^T\left(D_m\left(\phi\mu_f \nab \Qv\right):\nab\rv+\phi\mu_f \nab \Qv:D_m\left(\nab\rv\right)\right)\diff t+\\
	%---------------------------------------------
	&\int_0^TgD_m(\phi H)\nab z\cdot\rv\diff t+\int_0^Tg\phi HD_m\left(\nab z\right)\cdot\rv\diff t+\int_0^Tg\phi H\nab z\cdot\dot{\rv}\diff t-\\
	%---------------------------------------------
	&\int_0^T\frac{1}{2}gD_m(H^2)\nab \phi\cdot\rv\diff t-\int_0^T\frac{1}{2}g H^2D_m\left(\nab \phi\right)\cdot\rv\diff t-\int_0^T\frac{1}{2}gH^2\nab \phi
	\cdot\dot{\rv}\diff t+\\
	%--------------------------------------------- div terms
	& div(\Vv)\Big(\int_0^T-\frac{\partial p}{\partial t}\phi H\diff t+\phi H(T,x)p(T,x)-\phi H_0p(0,x)+\\
	%--------------------------------------------- 
	&\int_0^T-\frac{\partial \rv}{\partial t}\cdot(\phi\Qv)\diff t+ \phi\Qv(T,x)\cdot\rv(T,x)-\phi\Qv_0\cdot\rv(0,x)+\int_0^T \nab\Qv\cdot p\diff t+\\
	%---------------------------------------------
	&\int_0^T\mu_v\nab (H+z)\cdot\nab p\diff t+\int_{0}^T\nab\cdot\left(\phi\frac{\Qv}{H}\otimes \Qv\right)\cdot\rv\diff t+\\
	%---------------------------------------------
	&\int_0^T\frac{1}{2}g\nab(\phi H^2)\cdot\rv\diff t+\int_0^T\phi\mu_f \nab \Qv:\nab\rv\diff t+\int_0^Tg\phi H\nab z\cdot\rv\diff t-\\
	%---------------------------------------------
	&\int_0^Tg H^2\nab\phi\cdot\rv\diff t
	\Big)\Big]\diff x
	\end{align*}
	and the boundary part
	\begin{align*}
	&\int_{\Ga}\Big[\int_0^T[N(\hat{U}(t,x)-\bar{U}(t,x))]\cdot\dot{\hat{U}}\diff t+\\
	&div_{\Ga}(\Vv)\Big(\int_0^T[N(\hat{U}(t,x)-\bar{U}(t,x))]^2\diff t\Big)\Big]\diff s+\\
	%--------------------------------------------- Boundary Terms
	&\int_{\Gd}\Big[\int_0^T-\mu_v\nab (H_1+z) \cdot \nv \dot{p}\diff t+\\
	&div_{\Gd}(\Vv)\Big(\int_0^T-\mu_v\nab (H_1+z)\cdot \nv p\diff t\Big)\Big]\diff s+\\
	%---------------------------------------------
	&\int_{\Ge}\Big[\int_0^T [\![-\mu_vD_m\left(\nab (H+z)\right) \cdot \nv p-\mu_v\nab (H+z) \cdot \nv \dot{p}]\!]\diff t+\\
	%---------------------------------------------
	&\int_0^T[\![-\phi\mu_f D_m(\nab\Qv \cdot \nv)\cdot\rv-\phi\mu_f \nab\Qv \cdot \nv\cdot\dot{\rv}]\!]\diff t+\\
	%---------------------------------------------
	& div_{\Ge}(\Vv)\Big(\int_0^T[\![-\mu_v\nab(H+z)\cdot \nv p]\!]\diff t+\int_0^T[\![-\phi\mu_f \nab \Qv\cdot \nv\cdot\rv]\!]\diff t\Big)\Big]\diff s\text{.}
	\end{align*}
	The combination of both integrals, the non-commuting of material and spatial derivatives (\ref{Eq:11MatGradR}), (\ref{Eq:11MatGradRvec}) and (\ref{Eq:12MatGradProdR}), integration by parts combined with the fact that sediment and porosity move alongside with the deformation, which ultimately lets the material derivative vanish, such as finally regrouping for the material derivatives of the state $U=(H,\Qv)$ and adjoint variables $P=(p,\rv)$, lead to three parts, where firstly
	\allowdisplaybreaks
	\begin{align*}
	%--------------------------------------------- Part that vanishes due to adjoint for H derivative
	&\int_{\Ga}\int_0^T[N(\hat{U}(t,x)-\bar{U}(t,x))]\cdot\dot{U}\diff t\diff s+\int_\Om\int_0^T\Big[\\
	&\left(-\phi\frac{\partial p}{\partial t}+\frac{\phi}{H^2}(\Qv\cdot\nab)\rv\cdot\Qv-g\phi H(\nab\cdot\rv) -\nab\cdot(\mu_v\nab p)+g\phi\nab z\cdot \rv\right)\dot{H}+\\
	%--------------------------------------------- Part that vanishes due to adjoint for Q derivative
	&\left(-\phi\frac{\partial\rv}{\partial t}-\nab p-\frac{\phi}{H}(\Qv\cdot\nab)\rv-\frac{\phi}{H}(\nab\rv)^T\Qv-(\nab\cdot(\phi\mu_f \nab \rv))\right)\cdot\dot{\Qv}+\\	
	%--------------------------------------------- Part that vanishes due to state equation for mass conservation
	&\left(\phi\frac{\partial H}{\partial t}+\nab\cdot\left(\phi\Qv-\mu_v\nab (H+z)\right)\right)\dot{p}+\\
	%---------------------------------Part that vanishes due to state equation for moment conservation
	&\left(\phi\frac{\partial \Qv}{\partial t}+\nab\cdot\left(\phi\frac{\Qv}{H}\otimes \Qv+\frac{1}{2}g\phi H^2\mathbf{I}-\phi\mu_f \nab\Qv\right)+g\phi H\nab z\right)\cdot\dot{\rv}\diff t\Big]\diff x
	\end{align*}
	vanishes due to an evaluation the Lagrangian in its saddle point and secondly
		%--------------------------------------------- Boundary Terms
	\begin{align*}	
	&\int_{\Ga}\int_0^T\left[div_{\Ga}(\Vv)[N(\hat{U}(t,x)-\bar{U}(t,x))]^2\right]\diff t\diff s+\\
	%---------------------------------------------
	%---------------------------------------------
	&\int_{\Ge}\Big[\int_0^T\Big([\![\frac{\dot{\phi\Qv}}{H}\cdot\rv\Qv\cdot \nv+\frac{\phi\Qv}{H}\cdot\rv\dot{\Qv}\cdot \nv+p\dot{\phi\Qv}\cdot\nv+\\
	%---------------------------------------------
	& \int_0^T\frac{1}{2}gD_m(\phi H^2)\rv\cdot\nv
	]\!]\Big)\diff t\Big]\diff s+\\
	%---------------------------------------------
	& div_{\Ge}(\Vv)\Big(\int_0^T[\![-\mu_v\nab(H+z)\cdot \nv p]\!]\diff t+\int_0^T[\![-\phi\mu_f \nab \Qv\cdot \nv\cdot\rv]\!]\diff t\Big)\Big]\diff s
	\end{align*}
	vanishes since on the one hand outer boundaries are not variable and hence the deformation field $\Vv$ vanishes in small neighbourhoods around $\Ga, \Gd$ such that the material derivative is zero and on the other due the continuity of state and fluxes corresponding material derivatives are continuous. Finally, this leaves us with the shape derivative in its final form (\ref{Eq:21SD})
	\begin{align*}
	%---------------------------------Part that doesnt vanish for Shape Derivative
	DJ_1(\Om)[\Vv]=&\int_{0}^{T}\int_{\Om}\Big[-(\nab \Vv)^T:\nab (\phi\Qv) p - (\nab \Vv)^T:\nab \Qv\frac{\phi\Qv}{H}\cdot \rv- \\
	&(\nab \Vv\Qv\cdot\nab)\frac{\phi\Qv}{H}\cdot \rv -gH(\nab \Vv)^T\nab (\phi H) \cdot \rv-\\
	&\mu_v\nab (H+z)^T(\nab \Vv +\nab \Vv^T)\nab p - \\
	&\phi\mu_f \nab \Qv \nab \Vv:\nab \rv -\phi\mu_f \nab\Qv\nab \Vv^T:\nab\rv-  \\
	&g\phi H\nab \Vv^T\nab z\cdot\rv+\frac{1}{2}
	gH^2\nab \Vv^T\nab \phi\cdot\rv\\
	&div(\Vv)\{\frac{\partial \phi H}{\partial t}p+\nab\cdot (\phi\Qv) p+\frac{\partial \phi\Qv}{\partial t}\cdot \rv+\\
	& \phi(\Qv\cdot \nab)\frac{\Qv}{H}\cdot \rv +\nab\cdot (\phi\Qv)\frac{\Qv}{H}\cdot \rv +  \frac{1}{2}g\nab(\phi H^2)\cdot \rv +\\
	&g\phi H\nab z\cdot\rv+\mu_v\nab (H+z)\cdot\nab p +\\
	&\phi\mu_f \nab \Qv : \nab \rv -g\frac{1}{2}H^2\nab \phi\cdot\rv\}\Big]\diff x\diff t\\
	\end{align*}		
\end{proof}
\section{Derivation of DG Scheme for Interface Conditions}
\label{app:DGforInterface}
The porous SWE (\ref{Eq:SWE}) together with interface conditions on $\Ge$ can be resolved in an SIP-DG scheme.  
Starting from the weak form (\ref{Eq:17aweak}) and integrating by parts once more on the advective terms, in addition to once more using the jump identity $[\![ab]\!]=\{\!\{a\}\!\}[\![b]\!]+\{\!\{b\}\!\}[\![a]\!]$ together with flux continuity for the diffusive and advective flux we obtain 
\allowdisplaybreaks
\begin{equation}
\begin{aligned}
\allowdisplaybreaks
a(H,\Qv,p,\rv)=&\int_0^T\int_\Om-\frac{\partial p}{\partial t}\phi 
H\diff x\diff t+\int_\Om \phi\left[H(T,x)p(T,x)-H_0p(0,x)\right]\diff x+\\
%--------------------------------- flux sediment
&\int_0^T\int_\Om-\phi\Qv\cdot\nab p\diff x\diff t+\int_0^T\int_{\Gamma_{out}} p\phi\Qv\cdot \nv\diff s\diff t+\\
%--------------------------------- flux boundary interface
&\int_0^T\int_{\Ge} \{\!\{\phi\Qv\cdot\nv\}\!\} [\![p]\!]\diff s\diff t+\\
%--------------------------------- flux sediment
&\int_0^T\int_\Om\mu_v\nab (H+z)\cdot\nab p\diff x\diff t-\\
%--------------------------------- viscous sediment
&\int_0^T\int_{\Gamma_{out}}\left[p\mu_v\nab (H+z)\cdot \nv\right]\diff s\diff t-\\
%--------------------------------- viscous sediment interface
&\int_0^T\int_{\Ge}\left[\{\!\{\mu_v\nab(H+z)\cdot\nv\}\!\}[\![p]\!]\right]\diff s\diff t-\\
%--------------------------------- time momentum
&\int_0^T\int_\Om-\frac{\partial \rv}{\partial t}\cdot(\phi\Qv)\diff x\diff t+\int_\Om \phi\left[\Qv(T,x)\cdot\rv(T,x)-\Qv_0\cdot\rv(0,x)\right]\diff x+\\
%--------------------------------- 1. flux momentum source terms
&\int_0^T\int_\Om-\phi\frac{\Qv}{H}\cdot\nab\rv\cdot\Qv\diff x\diff t+\int_0^T\int_{\Gamma_{out}} \phi\frac{\Qv}{H}\cdot\rv\Qv\cdot \nv\diff s\diff t+\\
%--------------------------------- 2. flux momentum
&\int_0^T\int_\Om-\frac{1}{2}g\phi H^2\nab\cdot\rv\diff x\diff t+\int_0^T\int_{\Gamma_{out}} \frac{1}{2}g\phi H^2\rv\cdot \nv\diff s\diff t+\\
%--------------------------------- 2. flux momentum interface boundary
&\int_0^T\int_{\Ge}\{\!\{\phi(\frac{\Qv}{H}\otimes\Qv)\nv\}\!\} \cdot[\![\rv]\!]\diff s\diff t+\int_0^T\int_{\Ge}\{\!\{\frac{1}{2}g\phi H^2\nv\}\!\}\cdot[\![\rv]\!]\diff s\diff t+\\
%--------------------------------- 2. flux momentum
&\int_0^T\int_\Om \mu_f\phi\nab\Qv:\nab\rv\diff x\diff t+\int_0^T\int_\Om g\phi H\nab z \cdot\rv\diff x\diff t-\\
%--------------------------------- viscous momentum
&\int_0^T\int_\Gamma\left[\rv\cdot \mu_f\phi\nab \Qv\cdot \nv\right]\diff s\diff t-\\
%--------------------------------
&\int_0^T\int_{\Ge}\left[\{\!\{\mu_f\phi\nab\Qv\cdot\nv\}\!\}\cdot[\![\rv]\!]\right]\diff s\diff t-\\
%--------------------------------- 2. porous source
&\int_0^T\int_\Om g\frac{1}{2}H^2\nab\phi\cdot\rv\diff s\diff t\text{.}
\end{aligned}
\end{equation}

Since this derivation does not make use of the continuity of the solution (\ref{BCSWEfirst})-(\ref{BCSWEfirst_2}), we weakly enforce it by adding 
\begin{align}
\int_{0}^{T}\int_{\Gamma_3}\underline{\delta}(\hat{U}):\underline{[\![P]\!]}\diff s\diff t
\end{align}
for
\begin{equation}
\begin{aligned}
\underline{\delta}(\hat{U})=C_{IP}\frac{p^2}{h}\{\!\{G(f(\phi,\mu))\}\!\}\underline{[\![\hat{U}]\!]}\text{.}
\end{aligned}
\label{Eq:27viscousflux}
\end{equation}
In addition, it appears natural to symmetrize the diffusive part by
\begin{align}
-\int_{0}^{T}\int_{\Gamma_3}\{\!\{G(f(\phi,\mu))\nab (P)\}\!\}:\underline{[\![\hat{U}]\!]}
\end{align}
For the advective-flux we can refer to upwinding as
\begin{align}
\{\!\{\phi F(U)\cdot\nv\}\!\}_{Up} = \frac{1}{2}\left[\phi^+F(U^+)\cdot\nv+\phi^-F(U^-)\cdot\nv\right]\text{.}
\end{align}	
Finally a complete SIP-DG-scheme over $\mathcal{T}_h$ is obtained by allowing discontinuous cell-transitions, performing integration and integration by parts on each cell. If we allow alternation in the usage of the numerical flux function $\mathcal{F}$ we obtain the SIP-DG scheme in known form (\ref{Eq:22DGDiscretization}).

\section{Numerical Convergence of the Smoothed Approach}
\label{app:NumericalConvAlpha}
As mentioned in Section \ref{Sec:WellBalancedness} the numerical scheme used to handle discontinuous sediment and porosity coefficients forms the limit of a smoothed scenario. We numerically justify this by relying the smoothed porosity on smoothed step functions in one dimension, i.e. for discontinuities located at $x_0<x_1\in\R$ the smoothed porosity is obtained from
\begin{align}
\phi_\alpha(x)=\left[1-\psi(x,\alpha)\right]\phi_2+\psi(x,\alpha)\text{,}
\end{align}
where
\begin{align}
\psi(x,\alpha)=\begin{cases}
1& \text{ if } x\leq x_0-\alpha \wedge x\geq x_1+\alpha\\
-\frac{1}{4}\left(\frac{x_0-x}{\alpha}\right)^3+\frac{3}{4}\frac{x_0-x}{\alpha}+\frac{1}{2}&\text{ if } x>x_0-\alpha \wedge x< x_0+\alpha\\
1-\left(-\frac{1}{4}\left(\frac{x_1-x}{\alpha}\right)^3+\frac{3}{4}\frac{x_1-x}{\alpha}+\frac{1}{2}\right) &\text{ if } x>x_1-\alpha \wedge x< x_1+\alpha\\
0& \text{ if } x\geq x_0+\alpha \wedge x\leq x_1-\alpha\text{.}
\end{cases}
\end{align}
We can observe exemplifications for varying $\alpha$ in Figure \ref{fig:SmoothedPorosity}.
\begin{figure}[htb!]\centering
	\includegraphics[scale=0.36]{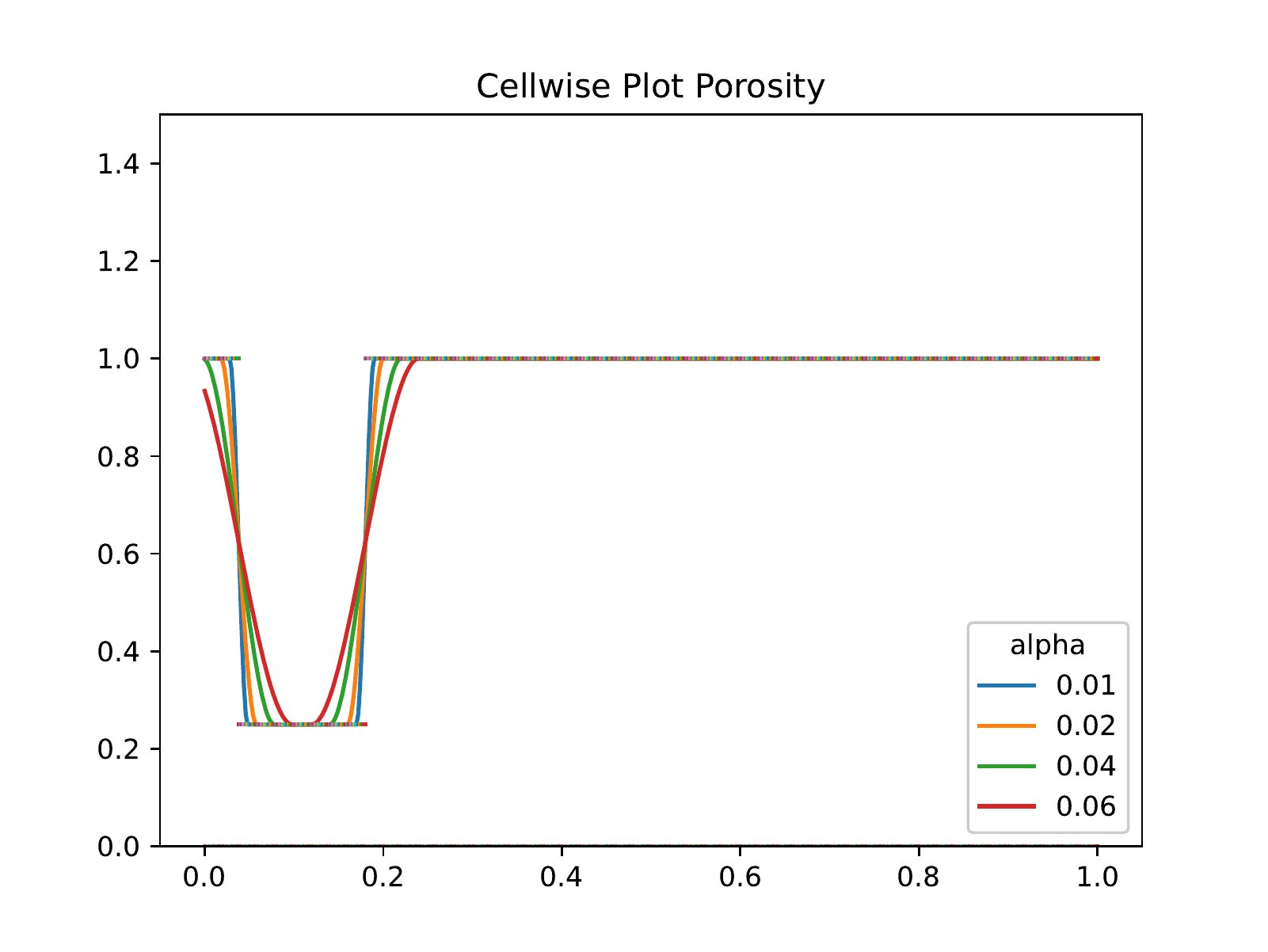}
	\caption{Smoothed Porosity for $\alpha\in\{0.01,0.02,0.04,0.06\}$}
	\label{fig:SmoothedPorosity}
\end{figure}
We now define error norms for water height $H$ and weighted velocity $uH$ as 
\begin{align}
E_H&=||H-H_\alpha||=\left(\int_0^T\int_\Om(H-H_\alpha)^2\diff x\diff t\right)^{1/2}\\
E_{uH}&=||uH-uH_\alpha||=\left(\int_0^T\int_\Om(uH-uH_\alpha)^2\diff x\diff t\right)^{1/2}\text{.}
\end{align}
From construction of the well-balanced scheme it is obvious that steady state conditions $uH=0$ for $H+z = C$ lead to zero error norms. We hence exemplifying investigate Gaussian initial conditions for the surface height $H$, i.e. $(H_0,uH_0)=(1+0.3\exp(-100(x-1/2)^2),0)$, for final time $T=0.4$ and step size $dt=\expnumber{1}{-3}$, where the discontinuities are located at $x_0=0.038$ and $x_1=0.18$.
We can observe the convergence numerically, i.e.  $U=\lim_{\alpha\rightarrow 0}U_\alpha$ for $\phi=\lim_{\alpha\rightarrow 0}\phi_\alpha$, as shown in table \ref{Tabel:Error}.
At this point, we would like to emphasise that the convergence is limited by the grid size of the mesh, hence showing the limit decrease for $\alpha\rightarrow0$ is only possible for $h_\kappa\rightarrow0$.
\begin{table}[htb!]
	\centering
	\begin{tabular}{l|c|c}
		$\alpha$ & $||H-H_\alpha||$  & $||uH-uH_\alpha||$\\ % <--
		\hline
		0.06 & 3.45587  & 8.41055\\
		0.04 & 2.05134 & 4.69693\\
		0.03 & 1.40056 &3.08506\\
		0.02 & 0.81980 & 1.69824\\
		0.01 & 0.3443 & 0.601958\\
		0.005 &0.17583  &0.23194\\
		0.001 &0.10549 &0.11108\\
	\end{tabular}
	\caption{Error Norms $E_H$ and $E_{uH}$ for decreasing $\alpha$}
	\label{Tabel:Error}
\end{table}





% \end{thebibliography}

\end{document}

%% file: objective_porous_3.tex
%auto-ignore
% This file was created with tikzplotlib v0.9.13.
\begin{tikzpicture}

\begin{axis}[
legend cell align={left},
legend style={fill opacity=0.8, draw opacity=1, text opacity=1, draw=white!80!black},
tick align=outside,
tick pos=left,
title={Objective Value},
x grid style={white!69.0196078431373!black},
xlabel={Iteration},
xmajorgrids,
xmin=-1.1, xmax=25.3,
xtick style={color=black},
y grid style={white!69.0196078431373!black},
ylabel={Objective},
ymajorgrids,
ymin=0.39922262253544, ymax=0.453332436404407,
ytick style={color=black}
]
\addplot [line width=1.64pt, blue]
table {%
0 0.450872899410363
1 0.447899562504334
2 0.444548792830951
3 0.440892586376256
4 0.437026859954027
5 0.433059202933991
6 0.429097402169315
7 0.425240981579931
8 0.421575299749916
9 0.418167593844178
10 0.41506526132783
11 0.41229669942244
12 0.409874183310384
13 0.407797642251068
14 0.406058269228965
15 0.404641396884818
16 0.403528573826704
17 0.40269899677845
18 0.40213051049692
19 0.401800337423172
20 0.401685639122874
21 0.401683600464817
22 0.401682159529484
23 0.401682015150845
24 0.401682012483124
};
\addlegendentry{$J(\Omega)$}
\end{axis}

\end{tikzpicture}

%% file: objective_porous_mentawai.tex
%auto-ignore
% This file was created with tikzplotlib v0.9.13.
\begin{tikzpicture}

\begin{axis}[
legend cell align={left},
legend style={fill opacity=0.8, draw opacity=1, text opacity=1, draw=white!80!black},
tick align=outside,
tick pos=left,
title={Objective Value},
x grid style={white!69.0196078431373!black},
xlabel={Iteration},
xmajorgrids,
xmin=-3.4, xmax=71.4,
xtick style={color=black},
y grid style={white!69.0196078431373!black},
ylabel={Objective},
ymajorgrids,
ymin=8.94329071172, ymax=11.877440160045,
ytick style={color=black}
]
\addplot [line width=1.64pt, blue]
table {%
0 11.744069730576
1 11.662870738685
2 11.582659567992
3 11.503471876765
4 11.425341741007
5 11.348302135842
6 11.272385178873
7 11.197622173577
8 11.124043383889
9 11.051677735856
10 10.980552428076
11 10.910692470094
12 10.842120308633
13 10.774855614386
14 10.708915301904
15 10.644313224908
16 10.581060310173
17 10.519164715586
18 10.458631991882
19 10.399465293377
20 10.341665613919
21 10.285232016131
22 10.230161841302
23 10.176451055444
24 10.124094333369
25 10.073085318469
26 10.02341682551
27 9.975081039036
28 9.928069706055
29 9.882374322625
30 9.837986308819
31 9.794897175597
32 9.75309867415
33 9.712582712919
34 9.673342011976
35 9.63536979001
36 9.598659781913
37 9.563206371386
38 9.529004545482
39 9.496049898323
40 9.464338607848
41 9.43386741076
42 9.404633584739
43 9.376634917021
44 9.34986976872
45 9.324336982957
46 9.30003605386
47 9.276967118812
48 9.255131022353
49 9.23452925609
50 9.215164055774
51 9.197038425786
52 9.18015617925
53 9.164521956098
54 9.150141198453
55 9.137020143848
56 9.125165795411
57 9.114585869562
58 9.105288717801
59 9.097283213999
60 9.09057861109
61 9.085184467781
62 9.081110304966
63 9.078365388669
64 9.076958406221
65 9.076897060902
66 9.07672651326
67 9.076661141189
68 9.076709099408
};
\addlegendentry{$J(\Omega)$}
\end{axis}

\end{tikzpicture}